\documentclass[12pt]{article}
\usepackage{amscd,amsmath,amssymb,amsfonts,braket,mathrsfs,latexsym,enumerate,amsthm}
\usepackage[all]{xy}

\usepackage{mathtools}

\makeatletter
\@namedef{subjclassname@2010}{%
  \textup{2010} Mathematics Subject Classification}
\makeatother

\newtheorem{thm}{Theorem}
\newtheorem{lem}[thm]{Lemma}

\newtheorem{cor}[thm]{Corollary}
\newtheorem{prop}[thm]{Proposition}

\theoremstyle{definition}
\newtheorem{remark}[thm]{Remark}
\newtheorem{df}[thm]{Definition}

\numberwithin{thm}{section}
\numberwithin{equation}{section}

\newcommand{\ZZ}{\mathbb{Z}} 
\newcommand{\QQ}{\mathbb{Q}} 
\newcommand{\CC}{\mathbb{C}} 
\newcommand{\FF}{\mathbb{F}} 

\newcommand{\cA}{\mathcal{A}}
\newcommand{\cB}{\mathcal{B}}

\newcommand{\cD}{\mathcal{D}}

\newcommand{\QQb}{\overline{\mathbb{Q}}}

\newcommand{\ol}[1]{\overline{#1}} 
\newcommand{\id}{\mathrm{id}} 
\newcommand{\Hom}{\mathrm{Hom}} 
\newcommand{\End}{\mathrm{End}} 
\newcommand{\MM}{\mathrm{M}} 
\newcommand{\Gal}{{\rm Gal}}

\newcommand{\into}{\hookrightarrow} 
\newcommand{\inv}{\mathrm{inv}} 

\newcommand{\Br}{{\rm Br}}

\newcommand{\fr}[1]{\mathfrak{#1}} 
\newcommand{\ca}[1]{\mathcal{#1}} 

\mathtoolsset{showonlyrefs=true}

\begin{document}

\title
{
An equivalent condition for abelian 
varieties over finite fields to have QM
}
\date{}

\author{Keisuke Arai, Yuuki Takai}

\date{}

\maketitle



\begin{abstract} 
In this paper, we give an equivalent condition for an abelian variety
over a finite field to have multiplication by a quaternion algebra over 
a number field. 
We prove the result by combining 
Tate's classification of the endomorphism algebras of abelian varieties
over finite fields
with Yu's criterion of the existence of homomorphisms
between semi-simple algebras.
\end{abstract}

\noindent
{\bf keywords:}
abelian varieties, QM, endomorphism algebras, quaternion algebras



\section{Introduction}

For abelian varieties over finite fields, Weil \cite{weil1948varietes},
Honda~\cite{honda1968isogeny}, and Tate \cite{tate1966endomorphisms} 
built a theory, so-called Honda-Tate theory, 
which gives a bijection between the isogeny classes 
of abelian varieties over a finite field $\FF_q$ and the conjugacy classes 
of some algebraic numbers (called Weil $q$-numbers), 
where $q$ is a power of a prime $p$. 
Furthermore, Tate determined the structures of the endomorphism algebras
of simple abelian varieties over finite fields. 

As a subsequent problem, conditions for abelian varieties 
to have multiplication of a given type have been investigated. 
In this direction, 
Jordan \cite[Theorem 2.1]{jordan1986points} 
determined all possible endomorphism algebras of abelian surfaces
over finite fields with quaternionic multiplication (or simply QM)
by an indefinite quaternion division algebra $B$ over $\QQ$, 
by using Waterhouse's classification of elliptic curves over finite fields
\cite[Theorem 4.1]{waterhouse1969abelian}.
As a consequence, he \cite[Proposition 2.3]{jordan1986points} 
gave a necessary condition for abelian surfaces over finite fields
to have QM by $B$.
Yu \cite[Theorem 1.1]{yu2013endo} 
extended these results to abelian surfaces
over any field, by using his work \cite[Theorem 1.2]{yu2012embeddings} 
on the existence of homomorphisms between semi-simple algebras.  

Jordan and the first author applied 
\cite[Theorem 2.1 and Proposition 2.3]{jordan1986points} 
to studies of rational points on
a Shimura curve associated with an indefinite quaternion division algebra 
$B$ over $\QQ$
(see \cite[Theorem 6.3]{jordan1986points},
\cite[Theorem 1.1]{Ara16}, and \cite[Theorems 2.3 and 2.4]{Ara18}).
Note that such a Shimura curve has a moduli interpretation over $\QQ$,
parameterizing isomorphism classes of QM-abelian surfaces by
(a maximal order of) $B$.
It is natural to consider an application of Jordan's argument
to a study of rational points on
a Shimura curve $M$ 
associated with a quaternion division algebra $B_0$ 
over a totally real number field $F_0$ of degree $g>1$.  
However, such a Shimura curve $M$ has no moduli interpretation over $F_0$. 
Then we should consider a Shimura curve $M'$ associated with 
$D=B_0\otimes_{F_0} E$ for a totally imaginary quadratic extension $E$ of $F_0$. 
The curve $M'$ has a moduli interpretation over $E$.
More precisely, $M'$ is a Shimura variety with PEL-structure 
consisting of an abelian variety of dimension $4g$, multiplication 
by an order of $D$, a polarization, and a level structure. 
For further details, see \cite{shimura1970canonical}, \cite{deligne1971travaux}, 
and \cite{morita1981reduction}. 

In this paper, based on this perspective, 
we give an equivalent condition for an abelian variety over $\FF_q$ 
of dimension $4g$ to have QM by a (not necessarily division) 
quaternion algebra $D$ over a number field $E$ of degree $2g$,
where $g\geq 1$.


Throughout this paper, we assume the following:
\begin{itemize}
\item
Every ring has $1$ (multiplicative identity)
and $1\ne 0$.
\item
Every ring homomorphism preserves $1$. 
\item
Every abelian variety over a field has positive dimension.
\item
``$A$ is a simple abelian variety over a field $k$''
means that $A$ is an abelian variety over $k$ which is simple over $k$.
\end{itemize}
To state our main theorem, we introduce some notation.
Let $A$ be an abelian variety over a field $k$, 
and $\End_k(A)$ its endomorphism ring over $k$.
Let 
\[
 \ca{E} =\End^0_k(A) = \End_k(A) \otimes_{\ZZ}\QQ
\]
be the endomorphism algebra of $A$ over $k$.
Note that $\ca{E}$ is a $\QQ$-algebra 
having $\id_A\otimes 1$ as a multiplicative identity.
Since $\End_k(A)$ is a finitely generated free $\ZZ$-module
(see \cite[\S 19, Theorem 3]{Mum70} or \cite[Theorem 12.5]{Mil86}),
the natural map
$\End_k(A)\rightarrow\ca{E}$
is injective.
Then we consider $\End_k(A)$ as a subring of $\ca{E}$.
Let $E$ be a number field of degree $2g$ where $g\geq 1$,
and $D$ a quaternion algebra 
over $E$. 
We do not assume that $E$ is totally imaginary or CM,
because such an assumption is unnecessary to prove the main theorem.
Nevertheless, we assume that $E$ is totally imaginary in
Corollary \ref{cor:main-intro-totally-real},
where the field generated by the Frobenius endomorphism is 
assumed to be totally real.
In this paper,
we say that $A$ has {\it quaternionic multiplication} by $D$
(or $A$ is an abelian variety with {\it quaternionic multiplication} by $D$) if  
there is a ring homomorphism
$$\imath \colon D \to \ca{E}.$$
In such a case, the map $\imath$ is injective since $D$ is a simple algebra.
Note also that $\imath$ is a $\QQ$-algebra homomorphism.
We often write {\it QM} instead of quaternionic multiplication for simplicity.
An abelian variety with QM shall be called a {\it QM-abelian variety}.

For abelian varieties $A, A'$ over $k$, we denote 
$$A\sim_k A'$$ 
if $A$ is isogenous  to $A'$ over $k$.
%
We know that
$A$ is decomposed as 
\begin{align}
\label{eq:A_prod_Ai}
A \sim_k \prod_{i=1}^{\ell} A_i^{m_i},
\end{align}
where $m_i\geq 1$ is an integer, 
$A_i$ is a simple abelian variety over $k$,
and $A_i\not\sim_k A_j$ if $i\neq j$;
the $\ell$, $m_i$'s are uniquely determined and the $A_i$'s are uniquely determined
up to isogeny over $k$
(see 
\cite[\S 12]{Mil86}).
We fix such a decomposition.
We say that $A$ is {\it iso-simple} over $k$ if 
$\ell=1$ in the decomposition \eqref{eq:A_prod_Ai}.
%
In \eqref{eq:A_prod_Ai}, 
each $A_i^{m_i}$ is called an {\it iso-simple factor} of $A$ over $k$.
Let
$$\ca{E}_i = \End_k^0(A_i)$$ 
for $i=1,\dots,\ell$.
Then each $\ca{E}_i$ is a division algebra.
We have
isomorphisms
\[
 \ca{E} 
 \simeq \prod_{i=1}^{\ell} \End_k^0(A_i^{m_i})
 \simeq \prod_{i=1}^{\ell} \MM_{m_i}(\ca{E}_i)
\]
of $\QQ$-algebras, which we often identify.
Here,
$\MM_{m_i}(\ca{E}_i)$ is the ring of $m_i\times m_i$ matrices
with entries in $\ca{E}_i$.
We consider $\ca{E}_i$ as a $\QQ$-subalgebra of
$\MM_{m_i}(\ca{E}_i)\simeq\End_k^0(A_i^{m_i})$
via the diagonal embedding.
Then the centers of $\ca{E}_i$ and $\MM_{m_i}(\ca{E}_i)$ coincide.
Note that $\ca{E}$ is semi-simple, and the center of $\ca{E}$
is the product of the centers of 
$\ca{E}_i$'s.
From the above description, we have:
\begin{lem}
\label{lem:simple_alg}
Let $A$ be an abelian variety over a field $k$.
Then the following conditions are equivalent:
\begin{enumerate}[\upshape(1)]
\item 
$A$ is iso-simple over $k$. 
\item 
$\End_k^0(A)$ is a simple algebra. 
\item 
The center of $\End_k^0 (A)$ is a field.
\end{enumerate}
\end{lem}

As for QM, we have the following lemma:
\begin{lem}
\label{lem:QM}
Let $A$ be an abelian variety over a field $k$, and decompose $A$ as in 
\eqref{eq:A_prod_Ai}.
Then the following conditions are equivalent:
\begin{enumerate}[\upshape(1)]
\item
$A$ has QM by $D$.
\item
$A_i^{m_i}$ has QM by $D$ for all $i\in\Set{1,\dots,\ell}$.
\end{enumerate}
\end{lem}

\begin{proof}
Assume that $A$ has QM by $D$.
Then there is a ring homomorphism
$\imath:D\rightarrow\ca{E}$.
Let 
$$pr_i:\ca{E} \simeq \prod_{i=1}^{\ell} \End_k^0(A_i^{m_i})
\rightarrow\End_k^0(A_i^{m_i})$$
be the $i$-th projection.
Then the composite map
$$pr_i\circ\imath:D\rightarrow\End_k^0(A_i^{m_i})$$
endows $A_i^{m_i}$ with QM by $D$.

Conversely, assume that there is a ring homomorphism
$\imath_i:D\rightarrow\End_k^0(A_i^{m_i})$
for each $i$.
Then the product map
$$\displaystyle\prod_{i=1}^\ell \imath_i:D
\rightarrow\prod_{i=1}^{\ell} \End_k^0(A_i^{m_i})\simeq\ca{E}$$
endows $A$ with QM by $D$.
\end{proof}

When $k$ is a finite field and $A$ is an abelian variety over $k$,
let 
$$\pi=\pi_A\in\End_k(A)\subseteq\ca{E}$$
be the Frobenius endomorphism of $A$ relative to $k$. 
Let $\QQ[\pi]\subseteq\ca{E}$ be the subalgebra generated 
by $\QQ$ and $\pi$.
Then $\QQ[\pi]$ is the center of $\ca{E}$
by \cite[Theorem 2 (a)]{tate1966endomorphisms}.
By Lemma \ref{lem:simple_alg}, the following conditions are equivalent:
\begin{itemize}
\item
$A$ is iso-simple over $k$.
\item
$\ca{E}$ is simple. 
\item
$\QQ[\pi]$ is a field.
\end{itemize}
In the decomposition in \eqref{eq:A_prod_Ai},
let 
\begin{itemize}
\item
$\pi_{A_i}\in\End_k(A_i)\subseteq\ca{E}_i$,
\item
$\pi_{A_i^{m_i}}\in\End_k(A_i^{m_i})
\subseteq\End_k^0(A_i^{m_i})\simeq\MM_{m_i}(\ca{E}_i)$
\end{itemize}
be the Frobenius endomorphisms relative to $k$ for $i=1,\dots,\ell$.
Then the diagonal embedding
$\ca{E}_i\subseteq\MM_{m_i}(\ca{E}_i)\simeq\End_k^0(A_i^{m_i})$
sends $\pi_{A_i}$ to $\pi_{A_i^{m_i}}$.
In this case, we identify and write
\begin{align}
\label{eq:pi}
\pi_{A_i}=\pi_{A_i^{m_i}}=\pi_i.
\end{align}
Then $\QQ[\pi_i]$ is a field, and it is a $\QQ$-subalgebra of $\ca{E}_i$.
We see that
$\QQ[\pi_i]$ is the center of $\ca{E}_i$ and $\MM_{m_i}(\ca{E}_i)$.
%
When $\QQ[\pi]$ (resp. $\QQ[\pi_i]$) is a field, 
we write $\QQ[\pi]=\QQ(\pi)$ (resp. $\QQ[\pi_i]=\QQ(\pi_i)$).
Then we have
\begin{align}
\label{eq:Q[pi]}
 \QQ[\pi] \simeq \prod_{i=1}^{\ell}\QQ(\pi_i)
\end{align}
inside
$\displaystyle\ca{E} 
\simeq \prod_{i=1}^{\ell} \MM_{m_i}(\ca{E}_i)$,
where $\pi$ corresponds to $(\pi_1,\dots,\pi_\ell)$.
When $\ell=1$, i.e., $A$ is iso-simple over $k$, we usually identify
$\pi=\pi_1$.
 
Our main theorem is as follows: 
\begin{thm}
\label{thm:main}
Let $A$ be an abelian variety of dimension $4g$ over $k=\FF_q$, 
let $E$ be a number field of degree $2g$, and let $D$ be a quaternion 
algebra over $E$. 
If $A$ has QM by $D$, 
then the number $\ell$ of iso-simple factors of $A$ over $k$ is $1$ or $2$.
Moreover, we have the following: 
\begin{enumerate}[\upshape(1)]
\item
Assume $\ell=1$. 
Then there are a simple abelian variety $A_1$ over $k$ 
and an integer $m_1\geq 1$ such that
\begin{itemize}
\item
$A\sim_k A_1^{m_1}$, and
\item
$\ca{E}_1 = \End^0_k(A_1)$ is a central division algebra over $\QQ(\pi)$.
\end{itemize} 
Let $c\geq 1$ be the integer such that $[\ca{E}_1:\QQ(\pi)]=c^2$.
In this case, $A$ has QM by $D$ if and only if one of the following holds: 
\begin{enumerate}
\item[\rm (i)]
There is an embedding $\jmath\colon \QQ(\pi)\into E$ such that 
		\begin{enumerate}
			\item[\rm (a)] 
			if $D \simeq \MM_2(E)$, 
			then $\ca{E}_1\otimes_{\QQ(\pi),\jmath} E \simeq \MM_c(E)$ holds, 
			\item[\rm (b)] 
			if $D \not\simeq \MM_2(E)$, then $c$ is even and 
				$\ca{E}_1\otimes_{\QQ(\pi),\jmath} E \simeq \MM_{c/2}(D)$ holds. 
		\end{enumerate}
		
\item[\rm (ii)] 
There is an embedding $\jmath\colon \QQ(\pi)\into E$ such that 
	\begin{enumerate}
		\item [\rm (a)] 
		if $D\simeq \MM_2(E)$, then $c$ is even and there is a 
			quaternion division algebra $D'$ over $E$ satisfying 
			$\ca{E}_1\otimes_{\QQ(\pi), \jmath}E \simeq \MM_{c/2}(D')$, 
		\item [\rm (b)] 
		if $D\not \simeq \MM_2(E)$, then either 
		 \begin{enumerate}
				\item[\rm 1.] 
				$\ca{E}_1\otimes_{\QQ(\pi), \jmath}E \simeq \MM_{c}(E)$ holds, or 
				\item[\rm 2.] 
				$c$ is even and there is a quaternion division algebra $D''$ over $E$
				satisfying $D''\not\simeq D$ and 
				$\ca{E}_1\otimes_{\QQ(\pi), \jmath}E \simeq \MM_{c/2}(D'')$.
		 \end{enumerate} 
	\end{enumerate} 
	
	\item[\rm (iii)] 
	There are a quadratic extension field $E'$ of $E$
	and an embedding $\jmath\colon \QQ(\pi) \into E'$ 
	such that $E'= E[\jmath(\pi)]$ and 
		\begin{enumerate}
			\item [\rm (a)] 
			if $D\otimes_{E}E' \simeq \MM_2(E')$, 
			then $\ca{E}_1\otimes_{\QQ(\pi),\jmath} E' \simeq \MM_{c}(E')$ holds, 
			\item [\rm (b)] 
			if $D\otimes_{E} E' \not\simeq \MM_2(E')$, then $c$ is even and 
		 $\ca{E}_1\otimes_{\QQ(\pi),\jmath} E' \simeq \MM_{c/2}(D\otimes_{E} E')$ 
		 holds. 
		\end{enumerate}
\end{enumerate}
Moreover, in the case {\rm (i)}, the integer $m_1$ is even.

\item
Assume $\ell=2$. 
Then there are simple abelian varieties $A_1,A_2$ over $k$ and integers
$m_1\geq 1,m_2\geq 1$ such that
\begin{itemize}
  \item
  $A_1\not\sim_k A_2$,
 	\item
 	$A \sim_kA_1^{m_1}\times A_2^{m_2}$, and
	\item 
	$\ca{E}_{i} = \End^0_k(A_i)$ is a central division algebra 
	over $\QQ(\pi_i)$ for $i=1,2$.
\end{itemize}
Let $c_i\geq 1$ be the integer such that 
$[\ca{E}_i:\QQ(\pi_i)]=c_i^2$ for $i=1,2$. 
In this case, $A$ has QM by $D$ if and only if both $A_1^{m_1}$ and $A_2^{m_2}$
have QM by $D$.
Moreover, for each $i$, $A_i^{m_i}$ has QM by $D$ if and only if 
\begin{enumerate}
\item[\rm (i)]
$\dim A_i^{m_i}=2g$, and
\item[\rm (ii)]
there is an embedding $\jmath_i\colon \QQ(\pi_i) \into E$ such that
\begin{enumerate}
 \item[\rm (a)] 
 if $D \simeq \MM_2(E)$, then
 $\ca{E}_i\otimes_{\QQ(\pi_i),\jmath_i} E \simeq \MM_{c_i}(E)$ holds,
 \item[\rm (b)] 
 if $D \not\simeq \MM_2(E)$, then
 $c_i$ is even and 
 $\ca{E}_i\otimes_{\QQ(\pi_i),\jmath_i} E \simeq \MM_{c_i/2}(D)$ holds.
\end{enumerate}
\end{enumerate}
\end{enumerate}
\end{thm}


This paper is organized as follows: 
As preliminaries, we review 
Tate's result on the structures of the endomorphism algebras
of abelian varieties over $\FF_q$
in Section \ref{sec:Tate},
and Yu's criterion of 
the existence of homomorphisms between semi-simple algebras
in Section \ref{sec:Yu}.
In Section \ref{sec:AV-of-small-dim}
(resp. Section \ref{sec:classification}),
we give an equivalent condition for an abelian variety over $\FF_q$
of dimension less than $4g$ (resp. of dimension $4g$)
to have QM by $D$.
In Section~\ref{subsec:totally-real}, 
we rephrase the equivalent condition for QM given in Sections
\ref{sec:AV-of-small-dim} and \ref{sec:classification},
under the assumption that the abelian variety is iso-simple,
$E$ is totally imaginary, and $\QQ(\pi)$ is totally real.
In Appendix~\ref{sec:abelian-surface}, 
we give an equivalent condition for an abelian surface over $\FF_q$ 
to have 
multiplication
by a quaternion algebra over $\QQ$, 
which is not necessarily indefinite or division. 


\section{Preliminaries}\label{sec:prelim}

\noindent
{\bf Notation and conventions}\\[-16pt]

\begin{itemize}
\item 
Let $g\geq 1$ be an integer. 
We fix a number field $E$ of degree $2g$ and a quaternion algebra 
$D$ over $E$. 
\item
Let $p$ be a prime and $r$ a positive integer.
Let 
$$q=p^r.$$
Let 
$$k=\FF_q$$ 
be the finite field of order $q$.   
\item 
For a number field $L$ and a place $w$ of $L$, 
let $L_w$ be the completion of $L$ at $w$.
\item 
Fix an algebraic closure $\ol{\QQ}$ of $\QQ$, and fix an embedding
$E\subseteq\ol{\QQ}$.
\item 
For a number field $L$, let $\Br(L)$ denote the Brauer group of $L$. 
For a central simple algebra $R$ over $L$, 
let $[R]$ denote the class of $R$ in $\Br(L)$.
For such $R$ and a place $w$ of $L$,
let $R^\circ$ be the opposite algebra of $R$,
and let $\inv_w(R)\in\QQ/\ZZ$
be the invariant of $R$ at $w$. 
\item
Let $|\cdot|$ be the standard absolute value on $\CC$.
\item
Let $\cD_p$ be the quaternion algebra over $\QQ$ which is ramified 
only at $p$ and $\infty$.
\item
Let $\cD'_p$ be the quaternion algebra over $\QQ(\sqrt{p})$ which is ramified 
only at the real places.
\end{itemize}

\subsection{Tate's classification of $\End_k^0(A)$}
\label{sec:Tate}
In this section, we review 
Tate's classification of the endomorphism algebra $\End_k^0(A)$.

\begin{df}[Weil $q$-number]
A Weil $q$-number is defined to be an algebraic integer $\Pi$ such 
that for every embedding $\tau \colon \QQ(\Pi) \into \CC$ of fields, we have 
\[
|\tau(\Pi)|=\sqrt{q}.
\]
We say that two Weil $q$-numbers $\Pi_1,\Pi_2$ 
are conjugate if there is an isomorphism
$\QQ(\Pi_1)\overset{\simeq}{\longrightarrow} \QQ(\Pi_2)$ 
of fields sending $\Pi_1$ to $\Pi_2$. 
\end{df}

For a simple abelian variety $A$ over $k=\FF_q$,
the Frobenius endomorphism $\pi=\pi_A$ is a Weil $q$-number
(see \cite[p.95]{tate1971classes}).
It is known that Weil $q$-numbers satisfy the following:  

\begin{prop}[{\cite[Proposition~4]{honda1968isogeny}}]
\label{prop:real-cm}
For any Weil $q$-number $\Pi$, the field $\QQ(\Pi)$ is
totally real or CM. 
\end{prop}

\begin{remark}
\label{HT}
Weil, Honda and Tate showed that the map 
which associates a simple abelian variety $A$ over $\FF_q$ 
to $\pi_A$ 
gives a bijection from the set of $\FF_q$-isogeny classes 
of simple abelian varieties over $\FF_q$ 
to the set of conjugacy classes of Weil $q$-numbers
(see \cite[Th\'{e}or\`{e}me 1 (i)]{tate1971classes}).
\end{remark}
%
%


We use the following theorem on the structure 
of the endomorphism algebra
of a simple abelian variety over $\FF_q$:

\begin{thm}[{\cite[p.142]{tate1966endomorphisms}, 
\cite[Th\'{e}or\`{e}me 1 (ii)]{tate1971classes}}]
\label{thm:Tate2}
Let $A$ be a simple abelian variety over $k=\FF_q$.
Let $\ca{E}=\End_k^0 (A)$, and let
$L=\QQ(\pi)\subseteq\ca{E}$.
Then we have the following: 

\begin{enumerate}
\item[{\rm (1)}]
$\ca{E}$ is a central division algebra over $L$. 
\item[{\rm (2)}]
$2 \dim A = [\ca{E}:L]^\frac{1}{2}[L:\QQ]$.
\item[{\rm (3)}]
As for local behaviors of $\ca{E}$, we have:
\begin{enumerate}
\item[\rm (a)]
The invariants of $\ca{E}$ at all real places of $L$ are $1/2$. 
\item[\rm (b)]
For every place $w$ of $L$ above $p$, we have
\[
 \inv_w(\ca{E}) = \frac{w(\pi)}{w(q)} [L_w: \QQ_p] 
 =\dfrac{w(\pi)}{r}f_w\bmod{\ZZ}.
\]
Here, we regard $w$ as a valuation of $L$ satisfying
$w(L^\times)=\ZZ$,
and $f_w$ is the degree of the residue field of $w$ over $\FF_p$
(the finite field of order $p$).
\item[\rm (c)]
$\ca{E}$ splits at all finite places of $L$ not above $p$.  
\end{enumerate}

\noindent
\item[{\rm (4)}]
Let $c\geq 1$ be the least common denominator of the invariants of $\ca{E}$ 
for all places of $L$. 
Then we have $[\ca{E}\colon L]=c^2$. 
\end{enumerate}
\end{thm}

%
Recall $q=p^r$.
When $A$ is simple and $\QQ(\pi)$ is totally real, 
as a corollary of Theorem~\ref{thm:Tate2}, 
we have $\dim A\leq 2$ and
the algebra $\End^0_k(A)$ is described more explicitly as follows: 

\begin{cor}
\label{cor:Tate-totally-real}
Let $A$ be a simple abelian variety over $k=\FF_q$, and let
$\ca{E}=\End_k^0 (A)$.
Assume that the field $\QQ(\pi)$ is totally real. 
Then we have the following: 
\begin{enumerate}[\upshape(1)]
	\item
	If $r$ is odd, then $A$ is an abelian surface, 
	$\QQ(\pi)\simeq \QQ(\sqrt{p})$, and 
	$\ca{E}\simeq \cD'_p$.
	\item
	If $r$ is even, then $A$ is an elliptic curve,  
	$\QQ(\pi)=\QQ$, and $\ca{E}\simeq \cD_p$.
\end{enumerate}
\end{cor}

\begin{proof}
Since $\pi$ is a Weil $q$-number,
we have
$|\tau(\pi)| = \sqrt{q}$
for any embedding
$\tau \colon \QQ(\pi)\into \CC$.
Then
$\tau(\pi)=\pm \sqrt{q}$
since $\QQ(\pi)$ is totally real.
This implies $\QQ(\pi) \simeq \QQ(\sqrt{q}) = \QQ(p^{r/2})$. 

Assume that $r$ is odd. 
Then $\QQ(\pi)\simeq\QQ(\sqrt{p})$. 
Since $p$ is ramified in $\QQ(\pi)\simeq \QQ(\sqrt{p})$, 
there is only one place (say $v'$) of $\QQ(\sqrt{p})$ dividing $p$. 
By Theorem~\ref{thm:Tate2}~(3), we have
\[
 \inv_v(\ca{E}) = 
 \begin{cases} 
 	1/2 & \text{if } v \text{ is a real place of } \QQ(\sqrt{p}), \\ 
 	 \frac{v(\pi)}{v(q)}[\QQ(\pi)_v :\QQ_p] 
 	 & \text{if } v=v',\\ 
 	0 & \text{otherwise.}
 \end{cases}
\]
Then
$$\inv_{v'}(\ca{E})=\dfrac{v'(\pm\sqrt{q})}{v'(q)}\cdot 2=1\equiv 0\bmod{\ZZ}.$$
Note that we can also compute $\inv_{v'}(\ca{E})$
by using
$\sum_v \inv_v(\ca{E}) = 0$.
By Theorem~\ref{thm:Tate2}~(4), we have 
$[\ca{E}:\QQ(\pi)]=4$.
Then $\ca{E}\simeq \cD'_p$.
By Theorem~\ref{thm:Tate2} (2), we have 
$$2\dim A = [\ca{E}:\QQ(\pi)]^{1/2}[\QQ(\pi):\QQ] = 4.$$
Hence $\dim A=2$.

Assume that $r$ is even. 
Then $\QQ(\pi) = \QQ$, and we have
\[
 \inv_v(\ca{E}) = \begin{cases} 
 	1/2 & \text{if } v \text{ is the infinite place of } \QQ, \\ 
 	\frac{v_p(\pi)}{v_p(q)} & \text{if $v$ is the $p$-adic valuation $v_p$ of $\QQ$,} \\ 
 	0 & \text{otherwise.}
 \end{cases}
\]
Then
$$\inv_{v_p}(\ca{E})=\dfrac{v_p(\pm\sqrt{q})}{v_p(q)}=\dfrac{1}{2}.$$
%
By Theorem~\ref{thm:Tate2}~(4), we have
$[\ca{E}:\QQ(\pi)]=4$.
Then $\ca{E}\simeq \cD_p$.
Thus
$$2\dim A = [\ca{E}:\QQ(\pi)]^{1/2}[\QQ(\pi):\QQ] = 2.$$
Hence $\dim A=1$.
\end{proof}

\subsection
{Homomorphisms between semi-simple algebras}
\label{sec:Yu}
To give an equivalent condition for an abelian variety over $\FF_q$
to have QM, we use the following criterion, 
which gives a necessary and sufficient condition 
for the existence of homomorphisms 
between semi-simple algebras:

\begin{thm}[{\cite[Theorem~1.2 (1)]{yu2012embeddings}}]
\label{thm:yu's_criterion}
Let $F$ be a field.
Let $\cA$ and $\cB$ be semi-simple $F$-algebras. 
We represent 
$$\cB\simeq\prod_{j=1}^s \End_{\Delta_j}(V_j),$$
where $\Delta_j$ is a division $F$-algebra 
and $V_j$ is a right $\Delta_j$-module ($j=1,\dots,s$).
%
For each $j$, we divide the maximal semi-simple quotient 
$(\Delta_j\otimes_F \cA^\circ)^{\mathrm{ss}}$ 
of $\Delta_j\otimes_F \cA^\circ$ into the product of simple factors: 
\[
 (\Delta_j\otimes_F \cA^\circ)^{\mathrm{ss}}
 \simeq\prod_{i=1}^{t_j}\MM_{m_{ji}}(D_{ji}), 
\] 
where $m_{ji}\geq 1$ is an integer and
$D_{ji}$ is a division $F$-algebra.
Then the set $\Hom_{F\text{\rm -alg}}(\cA,\cB)$ of $F$-algebra homomorphisms 
from $\cA$ to $\cB$ is non-empty 
if and only if there are non-negative integers 
$x_{ji}$ for $j=1,\dots,s$ and $i=1,\dots, t_j$ such that 
\begin{enumerate}
 \item[{\rm (I)}] 
 $\displaystyle\sum_{i = 1}^{t_j} x_{ji} = \dim_{\Delta_j} V_j$ \ for any $j$, and 
 \item[{\rm (II)}] 
 $\dfrac{m_{ji}[D_{ji}:F]}{[\Delta_j:F]}$ divides $x_{ji}$ for any $j,i$.
\end{enumerate}
\end{thm}

\begin{remark}
\label{rem:integer}
In the above theorem, we have:
\begin{enumerate}[\upshape(1)]
\item
If $\cA$ is simple, then
any homomorphism in $\Hom_{F\text{\rm -alg}}(\cA,\cB)$
is injective.
\item
If
$(\Delta_j\otimes_F \cA^\circ)^{\mathrm{ss}} 
\simeq\prod_{i=1}^{t_j}\MM_{m_{ji}}(D_{ji})$, 
then the rational number
$\dfrac{m_{ji}[D_{ji}:F]}{[\Delta_j:F]}$ is a positive integer for any $j,i$
(see \cite[p.4]{yu2012embeddings}).
\end{enumerate}
\end{remark}

\section{A proof of Theorem~\ref{thm:main}}

Let $A=A_1$ be a simple abelian variety over $k=\FF_q$.
Let
$\ca{E}_1=\End_k^0(A_1)$.
To apply Theorem \ref{thm:yu's_criterion} later,
we decompose
$\ca{E}_1\otimes_{\QQ}D^\circ$
as explained below.

Let 
$$h(T)=h_{A_1}(T)\in\QQ[T]$$
be the minimal polynomial of $\pi=\pi_A=\pi_{A_1}$.
Note that $h(T)$ is monic, irreducible,
$h(T)\in\ZZ[T]$,
and
$$\QQ(\pi)\simeq\QQ[T]/(h(T)).$$
There is an integer $c\geq 1$ such that
$[\ca{E}_1:\QQ(\pi)]=c^2$.
Let $d=[\QQ(\pi):\QQ]$.
Since $\QQ$ is of characteristic zero, 
$h(T)$ has no multiple root in $\QQb$.
We have
\begin{align}
\label{eq:hA1}
h(T)=\prod_{\tau:\QQ(\pi)\into\QQb}(T-\tau(\pi)),
\end{align}
where $\tau$ runs through all embeddings of $\QQ(\pi)$ into $\QQb$.
Decompose
\begin{align}
\label{eq:hA1'}
h(T)=\prod_{i=1}^t h_i(T),
\end{align}
where
$h_i(T)\in E[T]$
is a monic irreducible polynomial for $i=1,\dots,t$.
Note that $h_i(T)\ne h_j(T)$ if $i\ne j$.
Let
$$E_i=E[T]/(h_i(T))$$
for $i=1,\dots,t$.
Then each $E_i$ is a field, and we have
\begin{align}
\label{eq:QpiE}
\QQ(\pi)\otimes_\QQ E \simeq \QQ[T]/(h(T))\otimes_\QQ E 
\simeq E[T]/\Bigl(\prod_{i=1}^t h_i(T)\Bigr)
\simeq \prod_{i=1}^t E[T]/(h_i(T)) 
=\prod_{i=1}^t E_i. 
\end{align}
For $i=1,\dots,t$,
let
\begin{align}
\label{eq:ji}
\jmath_i:\QQ(\pi)\into E_i
\end{align}
be the embedding of fields defined by
$\pi\mapsto T\bmod{h_i(T)}$.
Then
\begin{align}
\label{eq:EiEpi}
E_i=E[\jmath_i(\pi)]=E(\jmath_i(\pi)).
\end{align}
We decompose
\begin{align} 
\label{eq:E1Dcirc}
 \ca{E}_1 \otimes_\QQ D^\circ &\simeq 
 \ca{E}_1 \otimes_{\QQ(\pi)} (\QQ(\pi) \otimes_\QQ E) \otimes_E D^\circ \simeq 
 \ca{E}_1 \otimes_{\QQ(\pi)} 
 \left(\prod_{i =1}^t E_i \right) \otimes_E D^\circ \\  
 & \simeq  \prod_{i = 1}^t ((\ca{E}_1 \otimes_{\QQ(\pi),\jmath_i} E_i) \otimes_{E_i} 
 (E_i \otimes_E D^\circ)).  
\end{align} 
Since $\ca{E}_1 \otimes_{\QQ(\pi),\jmath_i} E_i$ and 
$E_i \otimes_E D^\circ$ are central simple algebras over $E_i$, 
so is the tensor product 
$(\ca{E}_1 \otimes_{\QQ(\pi),\jmath_i} E_i) \otimes_{E_i} (E_i \otimes_E D^\circ)$
for each $i$
by \cite[Section~3, Corollaire~1]{serre1951}. 
Then there are a positive integer $n_i$ and a central division algebra $D_i$
over $E_i$ satisfying
\begin{align}
\label{eq:M-ni-Di}
(\ca{E}_1 \otimes_{\QQ(\pi),\jmath_i} E_i) \otimes_{E_i} 
(E_i \otimes_E D^\circ) 
\simeq \MM_{n_i}(D_i)
\end{align}
for each $i$.
In this case, 
$\dfrac{n_i[D_i:\QQ]}{[\ca{E}_1:\QQ]}$ is a positive integer 
by Remark~\ref{rem:integer} (2).
Note that
\begin{align}
\label{eq:nD/E}
\frac{n_i[D_i:\QQ]}{[\ca{E}_1:\QQ]} = \frac{2g n_i[D_i:E]}{c^2d},
\end{align}
because 
$[D_i:\QQ] = [D_i:E][E:\QQ]=2g[D_i:E]$ 
and 
$[\ca{E}_1:\QQ] = [\ca{E}_1:\QQ(\pi)][\QQ(\pi):\QQ] = c^2d$.

\subsection{QM-abelian varieties of dimension less than $4g$}
\label{sec:AV-of-small-dim}

In this section, 
we give an equivalent condition for an abelian variety over $k=\FF_q$
of dimension less than $4g$ to have QM by $D$.

\begin{prop}
\label{lem:smaller-dimension} 
Let $A_1$ be a simple abelian variety over $k=\FF_q$,
and let $m_1\geq 1$ be an integer.
Let $A$ be an abelian variety over $k$ such that $A\sim_k A_1^{m_1}$.
Let $\ca{E}_1=\End_k^0(A_1)$, $\ca{E}=\End_k^0(A)$,
and $c=[\ca{E}_1:\QQ(\pi)]^{1/2}$.
Assume $\dim A<4g$.
\begin{enumerate}[\upshape(1)]
\item
$A$ has QM by $D$ if and only if 
$\dim A=2g$ and
there is an embedding 
$\jmath\colon \QQ(\pi) \into E$ satisfying 
\[
 [\ca{E}_1 \otimes_{\QQ(\pi), \jmath} E] = [D]   
\]
in $\Br(E)$.
\item
Assume that there is an embedding
$\jmath\colon \QQ(\pi) \into E$.
\begin{enumerate}
\item[\rm (a)] 
When $D\simeq \MM_2(E)$, we have
$[\ca{E}_1 \otimes_{\QQ(\pi), \jmath} E] = [D]$ in $\Br(E)$
if and only if
$\ca{E}_1\otimes_{\QQ(\pi),\jmath} E \simeq\MM_{c}(E)$ holds. 
\item[\rm (b)] 
When $D\not\simeq \MM_2(E)$, we have
$[\ca{E}_1 \otimes_{\QQ(\pi), \jmath} E] = [D]$ in $\Br(E)$
if and only if $c$ is even and 
$\ca{E}_1\otimes_{\QQ(\pi),\jmath} E \simeq\MM_{c/2}(D)$ holds. 
\end{enumerate}
\end{enumerate}
\end{prop}

\begin{proof}
(1)
Recall $\pi=\pi_A=\pi_{A_1}$
as in \eqref{eq:pi}.
Let $d = [\QQ(\pi):\QQ]$.
By assumption, we have
$4g>\dim A=m_1\dim A_1$.
Then by Theorem~\ref{thm:Tate2} (2),
we have 
\begin{align}
\label{eq:cd-less-than-8g-m1}
 \frac{8g}{m_1} > 2\dim A_1 = [\ca{E}_1:\QQ(\pi)]^{1/2} [\QQ(\pi):\QQ] = cd. 
\end{align}
This implies
\begin{align}
\label{eq:8g-cd-larger-than-m1}
\frac{8g}{cd} > m_1. 
\end{align}
In the isomorphisms \eqref{eq:E1Dcirc} and \eqref{eq:M-ni-Di},
apply Theorem~\ref{thm:yu's_criterion} to 
$F=\QQ$, $\cA = D$, $s=1$, $\Delta_1 = \ca{E}_1$, $t_1 = t$, 
$V_1=\ca{E}_1^{m_1}$,
$\cB=\ca{E}\simeq\MM_{m_1}(\ca{E}_1)\simeq\End_{\Delta_1}(V_1)$,
$m_{1i}=n_i$,
$D_{1i}=D_i$.
Here, the isomorphism
$\MM_{m_1}(\ca{E}_1)\simeq\End_{\Delta_1}(V_1)$
holds since $V_1$ is a right $\Delta_1$-module
(see \cite[Theorem 7.2 (ii)]{Reiner}).
Then there is a ring homomorphism
$\imath\colon D \rightarrow \ca{E}$
if and only if
there are non-negative integers $x_1,\dots,x_t$
such that 
\begin{enumerate}
 \item[(I)] 
 $\displaystyle\sum_{i=1}^t x_i = m_1$, and
 \item[(II)] 
 $\displaystyle\frac{n_i[D_i:\QQ]}{[\ca{E}_1:\QQ]}$ divides $x_i$ for $i = 1, \dots, t$.  
\end{enumerate}

Assume that $A$ has QM by $D$, i.e.,
there is a ring homomorphism
$\imath\colon D \rightarrow \ca{E}$.
Then there are non-negative integers $x_1,\dots,x_t$
satisfying (I), (II).
By (II) and \eqref{eq:nD/E}, there are non-negative integers 
$k_1,\dots,k_t$
such that
\begin{align}
 x_i = \frac{2g n_i[D_i:E]}{c^2d}k_i \label{eq:relation-x-and-k}
\end{align}
for $i  =1, \dots, t$. 
By (I), \eqref{eq:8g-cd-larger-than-m1} and \eqref{eq:relation-x-and-k}, 
we have 
\begin{align}
 \frac{8g}{cd}>m_1 = \sum_{i=1}^t x_i = 
 \frac{2g}{c^2 d}\sum_{i=1}^t n_i[D_i:E]k_i. 
\end{align}
Then
\begin{align}
\label{eq:ineq-8g-cd-etc}
 4c > \sum_{i=1}^t n_i [D_i:E] k_i.
\end{align} 
On the other hand, by dimension counting for \eqref{eq:M-ni-Di}
over $E_i$, we have $4 c^2 = n_i^2 [D_i:E_i]$.
Then
\begin{align}
\label{eq:ni-c-Di-Ei} 
 n_i = \frac{2c}{[D_i:E_i]^{1/2}}. 
\end{align}
By \eqref{eq:ineq-8g-cd-etc} and \eqref{eq:ni-c-Di-Ei}, 
we have 
\[
 4c > \sum_{i = 1}^t n_i [D_i:E] k_i = 
 \sum_{i = 1}^t \frac{2c}{[D_i:E_i]^{1/2}} [D_i:E] k_i
 = 2c \sum_{i = 1}^t [D_i:E_i]^{1/2}[E_i:E] k_i. 
\]
Then
\begin{align}
\label{eq:ineq-Di-Ei-ki}
 2 > \sum_{i = 1}^t [D_i:E_i]^{1/2}[E_i:E] k_i. 
\end{align} 
If $\sum_{i = 1}^t [D_i:E_i]^{1/2}[E_i:E] k_i=0$,
then
$k_i=0$ holds for any $i$. 
This implies that $x_i = 0$ for any $i$.
This contradicts the condition (I): 
$\sum_{i=1}^t x_i =m_1\geq 1$.
Note that $[D_i:E_i]^{1/2}$ is a positive integer for each $i$
since $D_i$ is a central division algebra over $E_i$.
Then we have
$$\sum_{i = 1}^t [D_i:E_i]^{1/2}[E_i:E] k_i = 1.$$
This yields that 
$[D_i:E_i]=[E_i:E]=k_i=1$
for one $i$, and $k_j = 0$ 
for any $j\neq i$. 
Then $D_i = E_i = E$,
and the map
$\jmath_i:\QQ(\pi)\hookrightarrow E_i=E$
in \eqref{eq:ji}
is an embedding of $\QQ(\pi)$ into $E$.
By \eqref{eq:ni-c-Di-Ei}, we have
$n_i=2c$.
Then by \eqref{eq:M-ni-Di}, we have
\begin{align}
\label{eq:one-factor-of-product-quadratic}
(\ca{E}_1 \otimes_{\QQ(\pi), \jmath_i} E) \otimes_{E} D^\circ \simeq \MM_{2c}(E). 
\end{align}
Hence
\[
 [\ca{E}_1 \otimes_{\QQ(\pi), \jmath_i} E] = [D]
\]
holds in $\Br(E)$. 
By \eqref{eq:relation-x-and-k}, we have
$x_i=\dfrac{2g\cdot 2c\cdot 1}{c^2 d}\cdot 1=\dfrac{4g}{cd}$
and
$x_j = 0$ for any $j\neq i$. 
Then by (I), 
we have $m_1=x_i=4g/cd$. 
By \eqref{eq:cd-less-than-8g-m1}, we have 
$2\dim A_1=cd$.
Then
\[
 \dim A=m_1 \dim A_1=\frac{4g}{cd}\cdot \frac{cd}{2} = 2g. 
\]

Conversely, assume that $\dim A= 2g$ 
and that there is an embedding 
$\jmath\colon \QQ(\pi)\into E$ 
satisfying $[\ca{E}_1 \otimes_{\QQ(\pi), \jmath} E] = [D]$ 
in $\Br(E)$. 
Considering the composite map
$\QQ(\pi)\overset{\jmath}{\hookrightarrow}E\subseteq\QQb$,
we observe that
$T-\jmath(\pi)\in E[T]$
divides the polynomial
$$\displaystyle h(T)=\prod_{\tau:\QQ(\pi)\hookrightarrow\QQb}
(T-\tau(\pi))\in\QQ[T]\subseteq E[T]$$
in \eqref{eq:hA1}.
We have a monic irreducible decomposition
$\displaystyle h(T)=\prod_{i=1}^t h_i(T)$
in $E[T]$, where $h_i(T)\ne h_j(T)$ if $i\ne j$
(see \eqref{eq:hA1'}).
Then there is a unique
$i\in\Set{1,\dots,t}$
such that
$T-\jmath(\pi)=h_i(T)$.
Hence the $E$-algebra homomorphism
$E[T]\rightarrow E$
defined by
$T\mapsto\jmath(\pi)$
induces an isomorphism
$$\phi_i: E_i=E[T]/(h_i(T))=E[T]/(T-\jmath(\pi))
\overset{\simeq}{\longrightarrow} E$$
of fields over $E$.
Moreover, the following diagram is commutative:
\begin{equation*}
\begin{CD}
\QQ(\pi)@>\text{$\jmath_i$}>> E_i\\
@V\text{$\rm id$}VV@VV\text{$\phi_i$}V\\
\QQ(\pi)@>\text{$\jmath$}>> E.
\end{CD}
\end{equation*}
We may assume $i=1$
and identify $E_1=E$ by $\phi_1$.
Then 
$\jmath_1:\QQ(\pi)\hookrightarrow E_1$
is identified with
$\jmath:\QQ(\pi)\hookrightarrow E$.
Since
$[\ca{E}_1 \otimes_{\QQ(\pi), \jmath} E]-[D]=0$,
we have
$(\ca{E}_1\otimes_{\QQ(\pi), \jmath}E)\otimes_E D^\circ
\simeq \MM_n(E)$
for some $n\geq 1$.
By dimension counting over $E$, we have
$4c^2=n^2$.
Then $n=2c$, and so
$$(\ca{E}_1\otimes_{\QQ(\pi),\jmath_1}E_1)\otimes_{E_1}
(E_1\otimes_{E}D^\circ)
\simeq \MM_{2c}(E).$$
Comparing this with \eqref{eq:M-ni-Di},
we have
$n_1=2c$ and
$D_1\simeq E=E_1$.
We set 
$$k_1 = 1,\qquad k_i=0 \quad (i=2,\dots, t),$$
and 
\[
 x_i = \frac{2gn_i[D_i:E]}{c^2d}k_i \qquad (i=1,\dots, t).
\]
Then
$x_1=\dfrac{2g\cdot 2c\cdot 1}{c^2 d}\cdot 1=\dfrac{4g}{cd}$
and
$x_i=0$ for $i=2,\dots, t$.
Since $2\dim A_1 = cd$ 
and $\dim A=m_1 \dim A_1=2g$, 
we have 
$x_1 =\dfrac{4g}{cd}=\dfrac{2m_1\dim A_1}{2\dim A_1}=m_1$. 
Then the condition (I) is satisfied.
By \eqref{eq:nD/E}, we have
$x_i=\dfrac{n_i[D_i:\QQ]}{[\ca{E}_1:\QQ]}k_i$ for $i=1,\dots,t$.
Then the condition (II) is also satisfied.
Therefore $A$ has QM by $D$. 

(2)
Assume
$[\ca{E}_1 \otimes_{\QQ(\pi), \jmath} E] = [D]$ in $\Br(E)$.

(a)
If $D\simeq\MM_2(E)$, then
$\ca{E}_1 \otimes_{\QQ(\pi), \jmath} E\simeq\MM_n(E)$
for some $n\geq 1$.
By dimension counting over $E$, we have $c^2=n^2$.
Then $c=n$ and
$\ca{E}_1 \otimes_{\QQ(\pi), \jmath} E\simeq\MM_c(E)$.

(b)
If $D\not\simeq\MM_2(E)$, then
$\ca{E}_1 \otimes_{\QQ(\pi), \jmath} E\simeq\MM_n(D)$
for some $n\geq 1$.
By dimension counting over $E$, we have $c^2=4n^2$.
Then $c=2n$, which is even.
Moreover, we have
$\ca{E}_1 \otimes_{\QQ(\pi), \jmath} E\simeq\MM_{c/2}(D)$.

The converse is obvious.
\end{proof}

\begin{cor}
\label{is2g}
Let $A$ be an abelian variety over $k=\FF_q$.
Assume $\dim A<4g$.
If $A$ has QM by $D$, then $A$ is iso-simple over $k$
and $\dim A=2g$.
\end{cor}

\begin{proof}
Suppose that $A$ has QM by $D$.
Assume that $A$ is not iso-simple over $k$.
Then $A$ has at least two iso-simple factors $A',A''$ over $k$
such that
$0<\dim A'<\dim A$
and
$0<\dim A''<\dim A$.
Since $A$ has QM by $D$,
$A'$ and $A''$ also have QM by $D$
by Lemma \ref{lem:QM}.
Since $A'$ is iso-simple over $k$ and
$\dim A'<\dim A<4g$,
we have
$\dim A'=2g$
by Proposition \ref{lem:smaller-dimension} (1).
By the same argument, we have $\dim A''=2g$.
Then
$$\dim A\geq \dim A'+\dim A''=4g,$$
which contradicts the assumption $\dim A<4g$.
Therefore $A$ is iso-simple over $k$.
By Proposition \ref{lem:smaller-dimension} (1) again,
we have $\dim A=2g$.
\end{proof}

\subsection{QM-abelian varieties of dimension $4g$}
\label{sec:classification}

In this section, we give an equivalent condition for an abelian variety
over $k = \FF_q$ of dimension $4g$ to have QM by $D$,
and complete the proof of Theorem \ref{thm:main}.

\begin{lem}
\label{lem:l-at-most-2}
Assume that an abelian variety $A$ of dimension $4g$ over $k=\FF_q$ 
has QM by $D$.
Then $A$ has at most two iso-simple factors over $k$. 
\end{lem}

\begin{proof}
Decompose 
$A\sim_k \prod_{i=1}^\ell A_i^{m_i}$
as in \eqref{eq:A_prod_Ai}.
Assume $\ell\geq 3$.
Then for each $i$, we have
$\dim A_i^{m_i}<4g$.
Moreover, $A_i^{m_i}$ has QM by $D$ by Lemma \ref{lem:QM}.
Then by Proposition~\ref{lem:smaller-dimension} (1), 
we have
$\dim A_i^{m_i}=2g$
for each $i$.
This implies 
$4g=\dim A=2g\ell\geq 6g$,
which is a contradiction.
\end{proof}

For an iso-simple abelian variety over $k$ of dimension $4g$,
we have the following:

\begin{thm}
\label{thm:end-ring-E[pi]=E} 
Let $A_1$ be a simple abelian variety over $k=\FF_q$,
and let $m_1\geq 1$ be an integer.
Let $A$ be an abelian variety over $k$
such that $A\sim_k A_1^{m_1}$.
Assume $\dim A=4g$.
Then $A$ has QM by $D$ if and only if one of the following holds: 
\begin{enumerate}
	\item[\rm (i)]
	There is an embedding $\jmath\colon \QQ(\pi)\into E$ satisfying
	\[
	 [\ca{E}_1 \otimes_{\QQ(\pi), \jmath}E] = [D]
	\] 
	in $\Br(E)$. 
	\item[\rm (ii)]
	There are an embedding $\jmath\colon \QQ(\pi)\into E$ 
	and a quaternion division algebra $D'$ over $E$ satisfying 
	\[
	 [\ca{E}_1 \otimes_{\QQ(\pi), \jmath}E] - [D] = [D']
	\]
  in $\Br(E)$. 
	\item[\rm (iii)]
	There are a quadratic extension field $E'$ of $E$ 
	and an embedding $\jmath\colon \QQ(\pi) \into E'$ 
	satisfying $E'= E[\jmath(\pi)]$ and 
	\[
	 [\ca{E}_1\otimes_{\QQ(\pi),\jmath}E']=[D\otimes_{E}E'] 
	\]
	 in $\Br(E')$. 
\end{enumerate}
Moreover, in the case {\rm (i)}, the integer $m_1$ is even.
\end{thm}

\begin{proof}
We observe that
$\ca{E}_1 = \End_k^0(A_1)$ is a central division algebra 
over $\QQ(\pi)$ 
and $\ca{E} \simeq \MM_{m_1}(\ca{E}_1)$ holds.   
Let $d = [\QQ(\pi):\QQ]$ and $c = [\ca{E}_1:\QQ(\pi)]^{1/2}$. 
By Theorem~\ref{thm:Tate2}~(2), we have 
$2\dim A_1=cd$.
Since $\dim A=m_1\dim A_1$, we have
\begin{align}
\label{eq:m_1}
m_1=\dfrac{\dim A}{\dim A_1}=\dfrac{8g}{cd}.
\end{align}

In the isomorphisms \eqref{eq:E1Dcirc} and \eqref{eq:M-ni-Di},
apply Theorem~\ref{thm:yu's_criterion} to 
$F=\QQ$, $\cA = D$, $s=1$, $\Delta_1 = \ca{E}_1$, $t_1 = t$, 
$V_1=\ca{E}_1^{m_1}$,
$\cB=\ca{E}\simeq\MM_{m_1}(\ca{E}_1)\simeq\End_{\Delta_1}(V_1)$,
$m_{1i}=n_i$,
$D_{1i}=D_i$.
Then there is a ring homomorphism 
$\imath\colon D \to \ca{E}$
if and only if there are
non-negative integers $x_1,\dots,x_t$ such that 
\begin{enumerate}
 \item[(I)] 
 $\displaystyle\sum_{i=1}^t x_i = m_1$, and
 \item[(II)] 
 $\displaystyle \frac{n_i[D_i:\QQ]}{[\ca{E}_1:\QQ]}$ 
 divides $x_i$ for $i = 1, \dots, t$.  
\end{enumerate}

Assume that $A$ has QM by $D$, i.e.,
there is a ring homomorphism
$\imath\colon D \to \ca{E}$.
Then there are non-negative integers $x_1,\dots,x_t$
satisfying (I), (II).
By (II) and \eqref{eq:nD/E}, there are non-negative integers 
$k_1,\dots,k_t$
such that
\begin{align}
\label{eq:relation-x-and-k-2}
 x_i = \frac{2g n_i[D_i:E]}{c^2d}k_i 
\end{align}
for $i  =1, \dots, t$. 
By \eqref{eq:m_1} and (I), we have 
\[
 \frac{8g}{cd} =m_1= \sum_{i=1}^t x_i = 
 \frac{2g}{c^2 d}\sum_{i=1}^t n_i[D_i:E]k_i. 
\]
Then
\begin{align}
 4c = \sum_{i=1}^t n_i [D_i:E] k_i. \label{eq:c-ni-Di-ki}
\end{align} 
On the other hand, by dimension counting of \eqref{eq:M-ni-Di}
over $E_i$,
we have $4 c^2 = n_i^2 [D_i:E_i]$.
Then
\begin{align}
 n_i = \frac{2c}{[D_i:E_i]^{1/2}}. \label{eq:ni-c-Di-Ei-1}
\end{align} 
By \eqref{eq:c-ni-Di-ki} and \eqref{eq:ni-c-Di-Ei-1}, we have 
\[
 4c = \sum_{i = 1}^t n_i [D_i:E] k_i = 
 \sum_{i = 1}^t \frac{2c}{[D_i:E_i]^{1/2}} [D_i:E] k_i
 = 2c \sum_{i = 1}^t [D_i:E_i]^{1/2}[E_i:E] k_i. 
\]
This implies
\[
 2 = \sum_{i = 1}^t [D_i:E_i]^{1/2}[E_i:E] k_i.
\] 
Note that $[D_i:E_i]^{1/2}$ is a positive integer for each $i$.
Then we have one of the following: 
\begin{enumerate}[\upshape 1.]
	\item
	For one $i\in\Set{1,\dots,t}$, 
	we have $[D_i:E_i]=[E_i:E] = 1$, $k_i = 2$; and $k_j = 0$ for $j\neq i$.  
	\item
	For two distinct $i_1, i_2\in\Set{1,\dots,t}$, 
	we have $[D_{i_l}:E_{i_l}]=[E_{i_l}:E]=k_{i_l}=1$ where $l=1,2$; 
	and $k_j = 0$ for $j\neq i_1, i_2$.  
	\item
	For one $i\in\Set{1,\dots,t}$, 
	we have $[D_i:E_i]=4$, $[E_i:E]=k_i=1$; and $k_j = 0$ for $j\neq i$. 
	\item
	For one $i\in\Set{1,\dots,t}$, 
	we have $[D_i:E_i]=1$, $[E_i:E] = 2$, $k_i=1$; and $k_j = 0$ for $j\neq i$. 
\end{enumerate}

In the case 1, we have 
$D_i = E_i = E$.
Then there is an embedding
$\jmath_i\colon \QQ(\pi) \into E_i=E$
(see \eqref{eq:ji}).
We have
$n_i = 2c$
by \eqref{eq:ni-c-Di-Ei-1},
and
$(\ca{E}_1 \otimes_{\QQ(\pi),\jmath_i} E) 
\otimes_{E} D^\circ \simeq \MM_{2c}(E)$
by \eqref{eq:M-ni-Di}.  
Then 
$[\ca{E}_1 \otimes_{\QQ(\pi),\jmath_i} E]=[D]$ in $\Br(E)$. 
This implies that (i) holds. 

In the case 2, we have $D_{i_l} = E_{i_l} = E$ for $l=1,2$. 
Then there is an embedding
$\jmath_{i_l}\colon \QQ(\pi) \into E_{i_l}=E$
for $l=1,2$.
We have
$n_{i_l} = 2c$
and
$(\ca{E}_1 \otimes_{\QQ(\pi), \jmath_{i_l}} E) \otimes_{E} D^\circ 
\simeq \MM_{2c}(E)$.
Then $[\ca{E}_1 \otimes_{\QQ(\pi),\jmath_{i_l}} E]=[D]$ in $\Br(E)$. 
This also implies that (i) holds. 

In the case 3, we have $E_i=E$.
Then there is an embedding 
$\jmath_i:\QQ(\pi) \into E_i=E$.
We have $n_i = c$ 
and 
$(\ca{E}_1 \otimes_{\QQ(\pi),\jmath_i} E) \otimes_{E} D^\circ 
\simeq \MM_{c}(D_i)$.
Then
$[\ca{E}_1 \otimes_{\QQ(\pi),\jmath_i} E]-[D]=[D_i]$.
Since $D_i$ is a quaternion division algebra over $E_i=E$,
the condition (ii) holds.

In the case 4,
$D_i=E_i$ is a quadratic extension of $E$,
and there is an embedding
$\jmath_i:\QQ(\pi) \into E_i$
such that
$E_i=E[\jmath_i(\pi)]$
(see \eqref{eq:EiEpi}).
We have $n_i = 2c$ 
and 
\begin{align}
(\ca{E}_1 \otimes_{\QQ(\pi),\jmath_i} E_i) \otimes_{E_i} (E_i \otimes_E D^\circ) 
\simeq \MM_{2c}(E_i).
\end{align}
%
Then 
$[\ca{E}_1\otimes_{\QQ(\pi),\jmath_i}E_i] = [D\otimes_E E_i]$
in $\Br(E_i)$. 
This implies that (iii) holds.

In the following, we prove the converse. 
We first assume (i),
i.e., 
there is an embedding $\jmath \colon \QQ(\pi) \into E$ such that 
$[\ca{E}_1\otimes_{\QQ(\pi), \jmath} E]=[D]$ holds in $\Br(E)$. 
Then by the same argument as in the proof of
Proposition \ref{lem:smaller-dimension},
there is a unique
$i\in\Set{1,\dots,t}$
such that
$T-\jmath(\pi)=h_i(T)$.
Hence the $E$-algebra homomorphism
$E[T]\rightarrow E$
defined by
$T\mapsto\jmath(\pi)$
induces an isomorphism
$$\phi_i: E_i=E[T]/(h_i(T))=E[T]/(T-\jmath(\pi))
\overset{\simeq}{\longrightarrow} E$$
of fields over $E$.
Moreover, the following diagram is commutative:
\begin{equation*}
\begin{CD}
\QQ(\pi)@>\text{$\jmath_i$}>> E_i\\
@V\text{$\rm id$}VV@VV\text{$\phi_i$}V\\
\QQ(\pi)@>\text{$\jmath$}>> E.
\end{CD}
\end{equation*}
We may assume $i=1$ and
identify $E_1=E$ by $\phi_1$.
Then 
$\jmath_1:\QQ(\pi)\hookrightarrow E_1$
is identified with
$\jmath:\QQ(\pi)\hookrightarrow E$.
Since
$[\ca{E}_1 \otimes_{\QQ(\pi), \jmath} E]-[D]=0$,
we have
$(\ca{E}_1\otimes_{\QQ(\pi), \jmath}E)\otimes_E D^\circ
\simeq \MM_n(E)$
for some $n\geq 1$.
By dimension counting over $E$, we have $4c^2=n^2$.
Then $n=2c$, and so
$$(\ca{E}_1\otimes_{\QQ(\pi),\jmath_1}E_1)\otimes_{E_1}
(E_1\otimes_{E}D^\circ)
\simeq \MM_{2c}(E).$$
Comparing this with \eqref{eq:M-ni-Di},
we have
$n_1=2c$ and
$D_1\simeq E=E_1$.
We set 
$$k_1 = 2,\qquad k_i=0 \quad (i=2,\dots, t),$$
and 
\[
 x_i = \frac{2gn_i[D_i:E]}{c^2d}k_i \qquad (i=1,\dots, t).
\]
Then
$x_1=\dfrac{2g\cdot 2c\cdot 1}{c^2 d}\cdot 2=\dfrac{8g}{cd}=m_1$
by \eqref{eq:m_1},
and $x_i=0$ for $i=2,\dots,t$.
Hence the condition (I) is satisfied.
By \eqref{eq:nD/E}, we have
$x_i=\dfrac{n_i[D_i:\QQ]}{[\ca{E}_1:\QQ]}k_i$ for $i=1,\dots,t$.
Then the condition (II) is also satisfied.
Therefore $A$ has QM by $D$.
Moreover,
$\dfrac{m_1}{2}=\dfrac{x_1}{2}=\dfrac{n_1[D_1:\QQ]}{[\ca{E}_1:\QQ]}$
is an integer.
Then $m_1$ is even in this case.

Next, we assume (ii), 
i.e., there are an embedding $\jmath\colon \QQ(\pi)\into E$ 
and a quaternion division algebra $D'$ over $E$ satisfying 
$[\ca{E}_1 \otimes_{\QQ(\pi), \jmath}E] - [D] = [D']$ in $\Br(E)$.
Then there are 
$i\in\{1,\dots,t\}$
and an isomorphism
$\phi_i:E_i\to E$
such that the following diagram is commutative:
\begin{equation*}
\begin{CD}
\QQ(\pi)@>\text{$\jmath_i$}>> E_i\\
@V\text{$\rm id$}VV@VV\text{$\phi_i$}V\\
\QQ(\pi)@>\text{$\jmath$}>> E.
\end{CD}
\end{equation*}
We may assume $i=1$ and identify
$E_1=E$ by $\phi_1$.
Then 
$\jmath_1:\QQ(\pi)\hookrightarrow E_1$
is identified with
$\jmath:\QQ(\pi)\hookrightarrow E$.
Since
$[\ca{E}_1 \otimes_{\QQ(\pi), \jmath} E]-[D]=[D']$,
we have
$$(\ca{E}_1\otimes_{\QQ(\pi), \jmath}E)\otimes_E D^\circ
\simeq \MM_n(D')$$
for some $n\geq 1$.
By dimension counting over $E$, we have $n=c$.
Comparing this with \eqref{eq:M-ni-Di},
we have
$n_1=c$ and $D_1\simeq D'$.
We set 
$$k_1 = 1,\qquad k_i=0 \quad (i=2,\dots, t),$$
and 
\[
 x_i = \frac{2g n_i [D_i:E]}{c^2d}k_i \qquad (i=1,\dots, t).
\]
Then
$x_1=\dfrac{2g\cdot c\cdot 4}{c^2 d}\cdot 1=\dfrac{8g}{cd}=m_1$,
and $x_i=0$ for $i=2,\dots,t$.
We also have
$x_i=\dfrac{n_i[D_i:\QQ]}{[\ca{E}_i:\QQ]}k_i$ for $i=1,\dots,t$.
Then the conditions (I), (II) are satisfied.

Finally, we assume (iii),
i.e., there are a quadratic extension field $E'$ of $E$ 
and an embedding $\jmath\colon \QQ(\pi) \into E'$ 
satisfying $E'=E[\jmath(\pi)]$
and $[\ca{E}_1\otimes_{\QQ(\pi),\jmath}E']=[D\otimes_{E}E']$ 
in $\Br(E')$. 
The quadratic extension $E'/E$ is Galois, and let
$$g(T)=(T-\jmath(\pi))(T-\sigma(\jmath(\pi))),$$
where $\sigma\in\Gal(E'/E)$ is the non-trivial element.
Then
$g(T)\in E[T]$,
and $g(T)$ divides
$\displaystyle h(T)=\prod_{\tau:\QQ(\pi)\hookrightarrow\QQb}
(T-\tau(\pi))\in\QQ[T]\subseteq E[T]$.
%
Since $g(T)\in E[T]$ is monic and irreducible, 
there is a unique
$i\in\Set{1,\dots,t}$
such that
$g(T)=h_i(T)$ in the decomposition \eqref{eq:hA1'}.
Define an $E$-algebra homomorphism
$\phi'_i:E[T]\rightarrow E'$
by $T\mapsto\jmath(\pi)$.
It is surjective by the assumption $E'=E[\jmath(\pi)]$,
and we have
$\ker(\phi'_i)=(g(T))$.
Then it induces an isomorphism
$$\phi_i:E_i=E[T]/(h_i(T))=E[T]/(g(T))\overset{\simeq}{\longrightarrow}E'$$
of fields over $E$.
Moreover, 
the following diagram is commutative:
\begin{equation*}
\begin{CD}
\QQ(\pi)@>\text{$\jmath_i$}>> E_i\\
@V\text{$\rm id$}VV@VV\text{$\phi_i$}V\\
\QQ(\pi)@>\text{$\jmath$}>> E'.
\end{CD}
\end{equation*}
We may assume $i=1$ and
identify $E_1=E'$ by $\phi_1$.
Then 
$\jmath_1:\QQ(\pi)\hookrightarrow E_1$
is identified with
$\jmath:\QQ(\pi)\hookrightarrow E'$.
Since
$[\ca{E}_1 \otimes_{\QQ(\pi), \jmath}E']-[D\otimes_E E']=0$
in $\Br(E')$,
we have
$(\ca{E}_1\otimes_{\QQ(\pi), \jmath}E')\otimes_{E'} (D\otimes_E E')^\circ
\simeq \MM_n(E')$
for some $n\geq 1$.
By dimension counting over $E'$, we have
$n=2c$.
Then
$$(\ca{E}_1\otimes_{\QQ(\pi),\jmath_1}E_1)\otimes_{E_1}
(E_1\otimes_{E}D^\circ)
\simeq \MM_{2c}(E').$$
Comparing this with \eqref{eq:M-ni-Di},
we have
$n_1=2c$ and
$D_1\simeq E'=E_1$.
We set 
$$k_1 = 1,\qquad k_i=0 \quad (i=2,\dots, t),$$
and
\[
 x_i = \frac{2g n_i [D_i:E]}{c^2d}k_i \qquad (i=1,\dots, t).
\]
Then
$x_1=\dfrac{2g\cdot 2c\cdot 2}{c^2 d}\cdot 1=\dfrac{8g}{cd}=m_1$,
and $x_i=0$ for $i=2,\dots,t$.
We also have
$x_i=\dfrac{n_i[D_i:\QQ]}{[\ca{E}_i:\QQ]}k_i$ for $i=1,\dots,t$.
Then the conditions (I), (II) are satisfied.
\end{proof}

Now we are ready to prove Theorem \ref{thm:main}.

\begin{proof}[Proof of Theorem \ref{thm:main}]
If $A$ has QM by $D$, then $A$ has at most two iso-simple factors over $k$
by Lemma \ref{lem:l-at-most-2}.

(1)
Assume that $A$ is iso-simple over $k$, i.e., $\ell=1$.
We rephrase 
the equivalent conditions for QM in
Theorem \ref{thm:end-ring-E[pi]=E}.

In the case (i) of Theorem \ref{thm:end-ring-E[pi]=E}, 
there is an embedding
$\jmath:\QQ(\pi)\into E$
satisfying
\begin{align} 
 [\ca{E}_1\otimes_{\QQ(\pi),\jmath}E] = [D] \label{eq:E_1-D-1}
\end{align}
in $\Br(E)$.
Let $c=[\ca{E}_1:\QQ(\pi)]^{1/2}$.
If $D\simeq \MM_2(E)$, then
$\ca{E}_1\otimes_{\QQ(\pi),\jmath} E \simeq \MM_{c}(E)$. 
If $D\not\simeq \MM_2(E)$, then
$\ca{E}_1\otimes_{\QQ(\pi),\jmath} E \simeq \MM_{n}(D)$
for some $n\geq 1$.
By dimension counting over $E$, we have $c^2 = 4n^2$.
Then $c= 2n$, which is even.
Thus 
$\ca{E}_1\otimes_{\QQ(\pi),\jmath} E \simeq \MM_{c/2}(D)$. 
Therefore the condition (i) in Theorem \ref{thm:main} (1) holds.

Conversely, if the condition (i) in Theorem~\ref{thm:main} (1)
holds, then
$[\ca{E}_1\otimes_{\QQ(\pi),\jmath} E]=[D]$ in $\Br(E)$
in the cases (a), (b).
Thus the condition (i) in Theorem \ref{thm:end-ring-E[pi]=E} holds.

In the case (ii) of Theorem~\ref{thm:end-ring-E[pi]=E}, 
there are an embedding $\jmath\colon \QQ(\pi) \into E$
and a quaternion division algebra $D'$ over $E$ satisfying 
\begin{align}
[\ca{E}_1\otimes_{\QQ(\pi), \jmath} E] - [D] = [D'] 
\end{align}
in $\Br(E)$. 
Assume $D\simeq \MM_2(E)$. 
Then 
$[\ca{E}_1\otimes_{\QQ(\pi), \jmath} E] = [D']$.
By dimension counting over $E$, $c$ is even and 
$\ca{E}_1\otimes_{\QQ(\pi), \jmath} E\simeq \MM_{c/2}(D')$ holds. 
Assume $D\not\simeq \MM_2(E)$. 
Then 
$[\ca{E}_1\otimes_{\QQ(\pi),\jmath}E] = [D] + [D']$.
Since $D,D'$ are quaternion algebras over $E$,
we have
$2[D]=2[D']=0$.
Then 
$2[\ca{E}_1\otimes_{\QQ(\pi),\jmath}E]=0$.
By \cite[Remark 1.5.18]{Poonen}, there is a quaternion algebra $D''$ over $E$
such that
$[\ca{E}_1\otimes_{\QQ(\pi),\jmath}E]=[D'']$.
If $D''\simeq\MM_2(E)$,
then
$\ca{E}_1\otimes_{\QQ(\pi),\jmath}E\simeq\MM_c(E)$.
Assume $D''$ is division.
Then $c$ is even and
$\ca{E}_1\otimes_{\QQ(\pi), \jmath} E\simeq \MM_{c/2}(D'')$ holds. 
Since $D'$ is division, we have
$[D'']-[D]=[D']\ne 0$ in $\Br(E)$.
Then $D''\not\simeq D$.
Thus the condition (ii) in Theorem \ref{thm:main} (1) holds.

Conversely, when the condition (ii) in Theorem \ref{thm:main} (1)
holds, we find a quaternion division algebra $D'$ over $E$ satisfying
$$[\ca{E}_1\otimes_{\QQ(\pi), \jmath} E] - [D] = [D']$$
as follows.
In the case (ii) (a) of Theorem \ref{thm:main} (1),
let $D'$ be as in the statement of (ii) (a).
Then
$$[\ca{E}_1\otimes_{\QQ(\pi), \jmath} E] - [D] = [D']-0=[D'],$$
as required.
In the case (ii) (b) 1 of Theorem \ref{thm:main} (1), we have
$$[\ca{E}_1\otimes_{\QQ(\pi), \jmath} E] - [D] = 0-[D]=[D^\circ].$$
Let $D'=D^\circ$.
Since $D$ is division, so is $D'$.
In the case (ii) (b) 2 of Theorem \ref{thm:main} (1), 
let $D''$ be as in the statement of (ii) (b) 2.
Then
$$[\ca{E}_1\otimes_{\QQ(\pi), \jmath} E] - [D] = [D'']-[D].$$
Since
$2([D'']-[D])=0$,
there is a quaternion algebra $D'$ over $E$ such that
$[D'']-[D]=[D']$.
Since $D''\not\simeq D$,
we have $[D']\ne 0$.
Then $D'$ is division.
Therefore the condition (ii) in Theorem \ref{thm:end-ring-E[pi]=E} holds.

In the case (iii) of Theorem~\ref{thm:end-ring-E[pi]=E}, 
there are a quadratic extension field $E'$ of $E$ 
and an embedding $\jmath\colon \QQ(\pi) \into E'$ 
satisfying $E'=E[\jmath(\pi)]$ and 
\begin{align}
[\ca{E}_1\otimes_{\QQ(\pi),\jmath}E'] = [D\otimes_{E}E'] 
\label{eq:E_1-F-D-F} 
\end{align}
in $\Br(E')$. 
If $D\otimes_E E'\simeq\MM_2(E')$, then
$\ca{E}_1 \otimes_{\QQ(\pi),\jmath}E' \simeq \MM_c(E')$. 
If $D\otimes_{E}E'$ is division, then 
by dimension counting over $E'$, $c$ is even and
$\ca{E}_1\otimes_{\QQ(\pi),\jmath}E' \simeq \MM_{c/2}(D\otimes_{E}E')$ 
holds.
Therefore the condition (iii) in Theorem \ref{thm:main} (1) holds.

The converse is obvious.

(2)
Assume that $A$ has exactly two iso-simple factors over $k$, 
i.e., $\ell=2$.
Then the assertion follows from Lemma \ref{lem:QM}
and Proposition \ref{lem:smaller-dimension}.
\end{proof}

\subsection{The case where $\QQ[\pi]$ is a totally real field}
\label{subsec:totally-real}

When $A$ is iso-simple over $k=\FF_q$ and $\QQ(\pi)$ is totally real,
$\End_k^0(A)$ can be described explicitly by 
Corollary \ref{cor:Tate-totally-real}.
In this case, we give an equivalent condition for $A$ to have QM
by $D$, assuming further that $E$ is totally imaginary.
%

\begin{cor}
\label{cor:main-intro-totally-real}
Let $A_1$ be a simple abelian variety over $k=\FF_q$ with $q=p^r$,
and let $m_1\geq 1$ be an integer.
Let $A$ be an abelian variety over $k$ such that 
$A\sim_kA_1^{m_1}$.
Assume that $E$ is a totally imaginary field of degree $2g$,
and that $\QQ(\pi)$ is a totally real field.
Let 
$\ca{E}_1=\End_k^0(A_1)$.
%
\begin{enumerate}[\upshape(1)]
\item
Assume that $r$ is odd.
Then $\QQ(\pi)\simeq \QQ(\sqrt{p})$, 
$A_1$ is an abelian surface, 
and $\ca{E}_1\simeq \cD'_p$.
\begin{enumerate}
\item[{\rm (A)}] 
Suppose $\dim A=4g$.
Then $A$ has QM by $D$ if and only if one of the following holds:
\begin{enumerate}
\item[1.] 
There is an embedding $\jmath\colon \QQ(\pi)\into E$ 
such that 
$\ca{E}_1\otimes_{\QQ(\pi),\jmath}E \simeq \MM_2(E)$. 
\item[2.] 
There are a quadratic extension field $E'$ of $E$ 
and an embedding $\jmath\colon\QQ(\pi)\into E'$ 
such that $E'=E[\jmath(\pi)]$
and 
$D\otimes_{E}E'\simeq \MM_2(E')
\simeq \ca{E}_1\otimes_{\QQ(\pi),\jmath}E'$.
\end{enumerate}
\item[{\rm (B)}] 
Suppose $\dim A=2g$.
Then 
$A$ has QM by $D$ if and only if 
$D\simeq \MM_2(E)$ and
there is an embedding 
$\jmath\colon\QQ(\pi)\into E$ 
such that 
$\ca{E}_1\otimes_{\QQ(\pi),\jmath}E\simeq\MM_2(E)$.
\end{enumerate}
\item
Assume that $r$ is even.
Then $\QQ(\pi)=\QQ$, 
$A_1$ is an elliptic curve, 
and $\ca{E}_1\simeq \cD_p$.
\begin{enumerate}
\item[{\rm (A)}] 
Suppose $\dim A=4g$.
Then
$A$ has QM by $D$.  
\item[{\rm (B)}] 
Suppose $\dim A=2g$.
Then 
$A$ has QM by $D$ if and only if 
$\ca{E}_1\otimes_{\QQ} E \simeq D$ holds. 	
\end{enumerate}
\end{enumerate}
\end{cor}

\begin{proof}
Let
$c= [\ca{E}_1:\QQ(\pi)]^{1/2}$.

(1)
Assume that $r$ is odd.
Then $A_1$ is an abelian surface, 
$\QQ(\pi)\simeq\QQ(\sqrt{p})$,
and
$\ca{E}_1\simeq \cD'_p$
by Corollary~\ref{cor:Tate-totally-real}~(1).
Thus $c=2$. 
Since $\ca{E}_1$ is ramified only at the real places, we have
\begin{equation}\label{eq:splitting-by-CM}
\ca{E}_1\otimes_{\QQ(\pi)}E'' \simeq \MM_2(E'') 
\end{equation}
for any totally imaginary field $E''$ containing $\QQ(\pi)$ as a subfield.

(A)
Suppose $\dim A=4g$.

Assume that $A$ has QM by $D$.

In the case (i) (a) of Theorem~\ref{thm:main}~(1), 
we have
$D\simeq \MM_2(E)$ and
there is an embedding $\jmath\colon\QQ(\pi)\into E$ such that 
$\ca{E}_1\otimes_{\QQ(\pi),\jmath}E
\simeq\MM_c(E)=\MM_2(E)$. 
Then the condition 1 holds.
In the case (i) (b) of Theorem~\ref{thm:main}~(1), 
we have $D\not\simeq \MM_2(E)$ and
there is an embedding $\jmath\colon\QQ(\pi)\into E$ such that 
 $\ca{E}_1\otimes_{\QQ(\pi),\jmath}E\simeq D$. 
However, this contradicts \eqref{eq:splitting-by-CM}
since $E$ is totally imaginary by assumption.

In the case (ii) (a) of Theorem~\ref{thm:main}~(1),
we have $D\simeq \MM_2(E)$ and
there are an embedding $\jmath\colon\QQ(\pi)\into E$ 
and a quaternion division algebra $D'$ over $E$
such that 
$\ca{E}_1\otimes_{\QQ(\pi),\jmath}E\simeq D'$.
However, this contradicts \eqref{eq:splitting-by-CM}. 
In the case (ii) (b) 1 of Theorem~\ref{thm:main}~(1), 
we have
$D\not\simeq \MM_2(E)$ and 
there is an embedding $\jmath\colon\QQ(\pi)\into E$ 
such that 
$\ca{E}_1\otimes_{\QQ(\pi),\jmath}E \simeq \MM_2(E)$. 
Then the condition 1 holds.
In the case (ii) (b) 2 of Theorem~\ref{thm:main}~(1), 
we also have a contradiction to \eqref{eq:splitting-by-CM}. 

In the case (iii) (a) of Theorem~\ref{thm:main}~(1), 
the condition 2 holds since $c=2$.
In the case (iii) (b) of Theorem~\ref{thm:main}~(1), 
there are a quadratic extension field $E'$ of $E$
and an embedding $\jmath\colon \QQ(\pi)\into E'$ 
satisfying 
$E'=E[\jmath(\pi)]$,
$D\otimes_E E'\not\simeq \MM_2(E')$, 
and $\ca{E}_1\otimes_{\QQ(\pi),\jmath}E'\simeq D\otimes_E E'$.
This contradicts \eqref{eq:splitting-by-CM} because $E'$ is totally imaginary. 

Conversely, we prove that $A$ has QM by $D$ if the condition 1 or 2
holds.

Assume that the condition 1 holds.
If $D\simeq\MM_2(E)$, then
$$[\ca{E}_1\otimes_{\QQ(\pi),\jmath}E]=[D]$$
in $\Br(E)$.
Hence the condition (i) of Theorem~\ref{thm:end-ring-E[pi]=E}
is satisfied, and so $A$ has QM by $D$.
If $D\not\simeq\MM_2(E)$, let $D'=D^\circ$.
Then $D'$ is a quaternion division algebra over $E$ satisfying
$$[\ca{E}_1\otimes_{\QQ(\pi),\jmath}E]-[D]
=-[D]=[D^\circ]=[D'].$$
Hence the condition (ii) of Theorem~\ref{thm:end-ring-E[pi]=E}
is satisfied.

Next, assume that the condition 2 holds.
Then
$$[\ca{E}_1\otimes_{\QQ(\pi),\jmath}E']=[D\otimes_E E'],$$
and so the condition (iii) of Theorem~\ref{thm:end-ring-E[pi]=E}
is satisfied.

(B)
Suppose $\dim A=2g$.

Assume that $A$ has QM by $D$.
If $D\simeq\MM_2(E)$,
then by Proposition~\ref{lem:smaller-dimension},
there is an embedding 
$\jmath\colon\QQ(\pi)\into E$ 
such that 
$\ca{E}_1\otimes_{\QQ(\pi),\jmath}E\simeq\MM_c(E)=\MM_2(E)$. 
If $D\not\simeq\MM_2(E)$,
then there is an embedding 
$\jmath\colon\QQ(\pi)\into E$ 
such that
$\ca{E}_1\otimes_{\QQ(\pi),\jmath}E\simeq D$.  
However, this contradicts \eqref{eq:splitting-by-CM}
since $E$ is totally imaginary by assumption.

The converse follows from Proposition~\ref{lem:smaller-dimension} (1).

(2)
Assume that $r$ is even.
Then $A_1$ is an elliptic curve, 
$\QQ(\pi)=\QQ$, and $\ca{E}_1\simeq \cD_p$
by Corollary~\ref{cor:Tate-totally-real}~(2).
Thus $c=2$. 
Since $\QQ(\pi)=\QQ$, there is a unique embedding
$\jmath:\QQ(\pi)\hookrightarrow E$.

(A)
Suppose $\dim A=4g$.
We prove that $A$ has QM by $D$ in this case.

Assume $D\simeq\MM_2(E)$.
If $\ca{E}_1\otimes_\QQ E\simeq\MM_2(E)$,
then
$$[\ca{E}_1\otimes_{\QQ} E]=[D].$$
Hence the condition (i) of Theorem~\ref{thm:end-ring-E[pi]=E} is satisfied,
and so $A$ has QM by $D$.
If $\ca{E}_1\otimes_\QQ E\not\simeq\MM_2(E)$,
let $D'=\ca{E}_1\otimes_{\QQ} E$.
Then $D'$ is a quaternion division algebra over $E$,
and
$$[\ca{E}_1\otimes_{\QQ} E]-[D]=[D'].$$
Hence the condition (ii) of Theorem~\ref{thm:end-ring-E[pi]=E} is satisfied.

Next, assume
$D\not\simeq\MM_2(E)$.
If $\ca{E}_1\otimes_{\QQ} E\simeq\MM_2(E)$,
let
$D'=D^\circ$.
Then $D'$ is a quaternion division algebra over $E$,
and
$$[\ca{E}_1\otimes_{\QQ} E]-[D]=[D^\circ]=[D'].$$
Hence the condition (ii) of Theorem~\ref{thm:end-ring-E[pi]=E} is satisfied.
If $\ca{E}_1\otimes_\QQ E\not\simeq\MM_2(E)$,
let
$D''=\ca{E}_1\otimes_\QQ E$.
Then
$$[\ca{E}_1\otimes_{\QQ} E]-[D]=[D'']-[D].$$
Suppose $D''\simeq D$.
Then
$[\ca{E}_1\otimes_{\QQ} E]=[D]$,
and the condition (i) of Theorem~\ref{thm:end-ring-E[pi]=E} is satisfied.
Suppose $D''\not\simeq D$.
By \cite[Remark 1.5.18]{Poonen}, there is a quaternion division algebra $D'$
such that
$[D'']-[D]=[D']$.
Then
$$[\ca{E}_1\otimes_\QQ E]-[D]=[D'].$$
Hence the condition (ii) of Theorem~\ref{thm:end-ring-E[pi]=E}
is satisfied.

(B)
Suppose $\dim A=2g$.

Assume that $A$ has QM by $D$.
If $D\simeq\MM_2(E)$, then by Proposition~\ref{lem:smaller-dimension},
we have
$\ca{E}_1\otimes_\QQ E\simeq\MM_c(E)=\MM_2(E)$.
If $D\not\simeq\MM_2(E)$, 
then
$\ca{E}_1\otimes_\QQ E\simeq D$.
In any case, we have
$\ca{E}_1\otimes_\QQ E\simeq D$.

Conversely, if 
$\ca{E}_1\otimes_\QQ E\simeq D$,
then $A$ has QM by $D$ by Proposition~\ref{lem:smaller-dimension} (1).
\end{proof}

\appendix

\section{QM-abelian surfaces by a quaternion algebra over $\QQ$}
\label{sec:abelian-surface}

Let $A$ be an abelian surface over 
a field $F$.
Let $B$ be a quaternion algebra over $\QQ$. 
We do not assume that $B$ is indefinite or division.
In this appendix, we say that $A$ has 
{\it quaternionic multiplication} (or {\it QM}) by $B$
if there is a ring homomorphism 
$$\imath\colon B \to \ca{E}=\End_F^0(A).$$
Note that $\imath$ is injective since $B$ is a simple algebra.

In the following,
we provide an equivalent condition for an abelian surface over a finite field 
$k$ to have QM by $B$. 
This is an extension of the result
\cite[Proposition~2.3]{jordan1986points},
which gives a necessary condition 
for an abelian surface over $k$ to have QM 
by an indefinite quaternion division algebra over $\QQ$.
In \cite[Proposition~2.3]{jordan1986points},
the proof is done by using the classification
\cite[Theorem 2.1]{jordan1986points}
of the endomorphism algebras of QM-abelian surfaces.
Here, we use Theorem \ref{thm:yu's_criterion} in the proof instead.

\begin{thm}\label{thm:main-intro-elliptic}
Let $A$ be an abelian surface over $k=\FF_q$ with $q=p^r$, 
and let $B$ be a quaternion algebra over $\QQ$. 
\begin{enumerate}[\upshape(1)]
\item
Assume that $A$ is iso-simple over $k$, i.e., there are a simple 
abelian variety $A_1$ over $k$ and an integer $m_1\geq 1$
such that $A\sim_k A_1^{m_1}$.
Then $A$ has QM by $B$ if and only if one of the following holds:
\begin{enumerate}
\item[\rm (a)]
$A_1$ is an elliptic curve, $r$ is even, and $\End_k^0(A_1)\simeq \cD_p$.
\item[\rm (b)]
$A_1$ is an elliptic curve, 
$\End_k^0(A_1)=\QQ(\pi)$ is a quadratic field,
and
$\End_k^0(A)\simeq B\otimes_\QQ \QQ(\pi)\simeq\MM_2(\QQ(\pi))$.
\item[\rm (c)]
$A$ is simple over $k$, $\QQ(\pi)$ is a quadratic field,
and
$\End_k^0(A)\simeq B\otimes_\QQ \QQ(\pi)\not\simeq\MM_2(\QQ(\pi))$.
\end{enumerate}
\item
Assume that $A$ is not iso-simple over $k$, i.e.,
there are elliptic curves $A_1, A_2$ over $k$ such that
$A\sim_k A_1\times A_2$ 
and
$A_1\not\sim_k A_2$.
Then $A$ has QM by $B$ if and only if
\begin{enumerate}
\item[\rm (a)]
$r$ is even and
$B\simeq\End_k^0(A_1)\simeq\End_k^0(A_2)\simeq \cD_p$.
\end{enumerate} 
\end{enumerate}
\end{thm}


\begin{proof}
(1)
Let $\ca{E} = \End_k^0(A)$ and $\ca{E}_1 = \End_k^0(A_1)$. 
Then $\ca{E} \simeq \MM_{m_1}(\ca{E}_1)$ and 
$\ca{E}_1$ is a central division algebra over $\QQ(\pi)$.
We have
$2=\dim A=m_1\dim A_1$.
Let $c = [\ca{E}_1:\QQ(\pi)]^{1/2}$ and $d = [\QQ(\pi):\QQ]$. 
Then by Theorem~\ref{thm:Tate2}~(2), we have 
\begin{align}
\label{eq:m1-cd-surface}
\frac{4}{m_1}= 2\dim A_1=cd. 	
\end{align}
We consider the tensor product 
\begin{align}
\label{eq:E1B}
\ca{E}_1 \otimes_\QQ B^\circ \simeq 
\ca{E}_1 \otimes_{\QQ(\pi)} (\QQ(\pi) \otimes_\QQ B^\circ). 
\end{align}
Since both $\ca{E}_1$ and $\QQ(\pi)\otimes_\QQ B^\circ$ 
are central simple algebras over $\QQ(\pi)$,
so is
$\ca{E}_1 \otimes_{\QQ(\pi)} (\QQ(\pi) \otimes_\QQ B^\circ)$.
Then there are a positive integer $n_1$ and a central division algebra $D_1$
over $\QQ(\pi)$ satisfying
\begin{align}
\label{eq:E1Bcirc}
\ca{E}_1 \otimes_{\QQ(\pi)} (\QQ(\pi) \otimes_\QQ B^\circ)
\simeq\MM_{n_1}(D_1).
\end{align}
In this case, 
$n_1[D_1:\QQ]/[\ca{E}_1:\QQ]$ is a positive integer
by Remark \ref{rem:integer} (2).

In the isomorphisms \eqref{eq:E1B} and \eqref{eq:E1Bcirc},
apply Theorem~\ref{thm:yu's_criterion} to 
$F=\QQ$, 
$\cA = B$, 
$s=1$, 
$\Delta_1 = \ca{E}_1$, 
$t_1 = 1$, 
$V_1=\ca{E}_1^{m_1}$,
$\cB=\ca{E}\simeq\MM_{m_1}(\ca{E}_1)\simeq\End_{\Delta_1}(V_1)$, 
$m_{11}=n_1$, $D_{11}=D_1$.
Then there is a ring homomorphism 
$\imath\colon B \rightarrow \ca{E}$
if and only if there is a non-negative integer $x_1$ such that 
\begin{enumerate}
 \item[(I)] 
 $x_1 = m_1$, and
 \item[(II)] 
 $\displaystyle \frac{n_1[D_1:\QQ]}{[\ca{E}_1:\QQ]}$ divides $x_1$.  
\end{enumerate}
Note that
$\displaystyle \frac{n_1[D_1:\QQ]}{[\ca{E}_1:\QQ]} = \frac{n_1[D_1:\QQ]}{c^2d}$.

Assume that $A$ has QM by $B$, i.e.,
there is a ring homomorphism 
$$\imath: B\rightarrow\ca{E}.$$
Then there is a non-negative integer $x_1$ satisfying (I), (II).
By (I) and (II), there is a non-negative integer $k_1$ satisfying
\begin{align}
\label{eq:m1-x1-k1-surface}
 m_1 = x_1 = \frac{n_1[D_1:\QQ]}{c^2d}k_1.
\end{align}
By \eqref{eq:m1-cd-surface} and~\eqref{eq:m1-x1-k1-surface}, we have 
$4 = cd m_1 = \dfrac{n_1[D_1:\QQ]}{c}k_1$.
Then
\begin{align}
\label{eq:4c-n1D1k1-surface}
 4c = n_1[D_1:\QQ]k_1. 
\end{align}
On the other hand, by dimension counting for \eqref{eq:E1Bcirc}
over $\QQ(\pi)$,
we have 
$4c^2 =n_1^2[D_1:\QQ(\pi)]$.
Then
\begin{align}
\label{eq:n1-2c-D1-surface}
n_1=\frac{2c}{[D_1:\QQ(\pi)]^{1/2}}. 
\end{align}
By \eqref{eq:4c-n1D1k1-surface} and \eqref{eq:n1-2c-D1-surface}, 
we have 
\[
 4c = \frac{2c}{[D_1:\QQ(\pi)]^{1/2}}[D_1:\QQ]k_1 
 = 2c [D_1:\QQ(\pi)]^{1/2}[\QQ(\pi):\QQ]k_1.
\]
This implies
\begin{align}
\label{eq:2-D1Qpik1-surface}
2 = [D_1:\QQ(\pi)]^{1/2}[\QQ(\pi):\QQ]k_1. 
\end{align}
Note that $[D_1:\QQ(\pi)]^{1/2}$ is a positive integer
since $D_1$ is a central division algebra over $\QQ(\pi)$.
Then we have one of the following: 
\begin{enumerate}
	\item[1.] 
	$[D_1:\QQ(\pi)]=1$, $[\QQ(\pi):\QQ]=1$, $k_1=2$.
	\item[2.] 
	$[D_1:\QQ(\pi)]=4$, $[\QQ(\pi):\QQ]=1$, $k_1 = 1$.
	\item[3.] 
	$[D_1:\QQ(\pi)]=1$, $[\QQ(\pi):\QQ]=2$, $k_1=1$.
\end{enumerate}

In the case 1, we have
$D_1=\QQ(\pi)=\QQ$.
Then $d=1$.
By \eqref{eq:n1-2c-D1-surface}, we have
$n_1=2c$.
Then
\begin{align}
\ca{E}_1 \otimes_\QQ B^\circ \simeq \MM_{2c}(\QQ)
\end{align}
by \eqref{eq:E1B} and \eqref{eq:E1Bcirc}.
Thus $[\ca{E}_1]=[B]$ in $\Br(\QQ)$. 
By Corollary~\ref{cor:Tate-totally-real},
$\QQ(\pi)=\QQ$ implies that
$r$ is even, $\ca{E}_1\simeq \cD_p$
and
$A_1$ is an elliptic curve.
Then $c =[\ca{E}_1:\QQ(\pi)]^{1/2} = 2$
and
$m_1=2$. 
Thus the condition (a) holds.
Note that
$B\simeq\ca{E}_1\simeq \cD_p$
in this case.

In the case 2, we have
$\QQ(\pi)=\QQ$ and $D_1$ is a quaternion division algebra over $\QQ$. 
Then $d=1$, $n_1=c$ and
\begin{align}
 \ca{E}_1 \otimes_\QQ B^\circ \simeq 
 \MM_{c}(D_1). 
\end{align}
Thus
$[\ca{E}_1]-[B]=[D_1]$
in $\Br(\QQ)$.
By Corollary~\ref{cor:Tate-totally-real},
$\QQ(\pi)=\QQ$ implies that
$r$ is even, $\ca{E}_1\simeq \cD_p$
and $A_1$ is an elliptic curve.
Then
$c=2$ and $m_1=2$.  
Thus the condition (a) holds.
Note that 
$B\not\simeq\ca{E}_1\simeq \cD_p$
in this case.

In the case 3, we have
$D_1=\QQ(\pi)$,
which is a quadratic field.
Then $d=2$, $n_1 = 2c$ and
\begin{align}
\ca{E}_1 \otimes_{\QQ(\pi)} (\QQ(\pi) \otimes_\QQ B^\circ) 
\simeq \MM_{2c}(\QQ(\pi)). 
\end{align}
Thus
\begin{align}
\label{eq:E1BQpi}
[\ca{E}_1] = [B\otimes_\QQ \QQ(\pi)]
\end{align}
in $\Br(\QQ(\pi))$. 
By \eqref{eq:m1-cd-surface}, we have
$4/m_1 = cd = 2c$.
Then $2 = cm_1$. 
Hence $(c,m_1)=(1,2)$ or $(2,1)$.
Assume $(c,m_1)=(1,2)$. 
Then 
$A_1$ is an elliptic curve, 
and $\ca{E}_1=\QQ(\pi)$. 
By \eqref{eq:E1BQpi}, we have
$B\otimes_{\QQ}\QQ(\pi) \simeq \MM_2(\QQ(\pi))\simeq\ca{E}$. 
Then the condition (b) holds. 
Assume $(c,m_1)=(2,1)$. 
Then 
$A_1$ is an abelian surface, 
and $\ca{E}_1$ is a quaternion division algebra over $\QQ(\pi)$. 
Thus
$\ca{E}\simeq\ca{E}_1 \simeq B\otimes_\QQ \QQ(\pi)$ 
by \eqref{eq:E1BQpi}. 
Hence the condition (c) holds. 

We prove the converse.  
Assume that the condition (a) holds.
Then
$\ca{E}_1\simeq \cD_p$ and
$m_1 = 2$. 
Comparing the centers of
$\ca{E}_1\simeq \cD_p$,
we have $\QQ(\pi)=\QQ$.
Assume $B\simeq \ca{E}_1$. 
Then $\ca{E}\simeq\MM_{m_1}(\ca{E}_1)\simeq \MM_2(B)$ holds. 
Thus $A$ has QM by $B$
by the diagonal embedding
$B\hookrightarrow\MM_2(B)\simeq\ca{E}$.
Next, assume $B \not \simeq \ca{E}_1$. 
Then there is a quaternion division algebra $D'$ 
over $\QQ$ such that 
$\ca{E}_1\otimes_\QQ B^\circ \simeq \MM_2(D')$
by \cite[Remark 1.5.18]{Poonen}.
This implies that $D_1\simeq D'$ and $n_1 = 2$
in \eqref{eq:E1Bcirc}.
We set 
$$x_1=2.$$
Then $x_1 = m_1 = 2$ and 
\[
 \frac{n_1[D_1:\QQ]}{[\ca{E}_1:\QQ]} 
 = \frac{2\cdot 4}{4} = 2 \mid x_1. 
\]
Thus (I), (II) hold, and so $A$ has QM by $B$. 

Assume that the condition (b) or (c) holds.
Since 
$B\otimes_\QQ \QQ(\pi)\simeq\ca{E}$,
the map
$B\rightarrow B\otimes_\QQ \QQ(\pi)$
defined by
$b\mapsto b\otimes 1$
endows $A$ with QM by $B$.

%

(2)
Let
$\ca{E}_i=\End_k^0(A_i)$
for $i=1,2$.
Then $\ca{E}\simeq \ca{E}_1\times \ca{E}_2$. 

Assume that $A$ has QM by $B$, i.e., there is a ring homomorphism
$\imath\colon B\to \ca{E}$. 
Then consider
the composition of the ring homomorphisms
\[
\imath_i:B\overset{\imath}{\longrightarrow}\ca{E} 
\simeq \ca{E}_1\times \ca{E}_2 
\overset{pr_i}{\longrightarrow}\ca{E}_i
\]
for $i=1,2$, where $pr_i$ is the $i$-th projection.
%
Note that each $\imath_i$ is injective
since $B$ is a simple algebra.
By Theorem~\ref{thm:Tate2} (2), we have
$$2=[\ca{E}_i:\QQ(\pi_i)]^{1/2}[\QQ(\pi_i):\QQ]
=[\ca{E}_i:\QQ]^{1/2}[\QQ(\pi_i):\QQ]^{1/2}
\geq [\ca{E}_i:\QQ]^{1/2}.$$
Then
$[\ca{E}_i:\QQ]\leq 4$.
Since
$[B:\QQ]=4$, the $\QQ$-algebra homomorphism
$$\imath_i:B\rightarrow\ca{E}_i$$
is an isomorphism.
Thus the center $\QQ(\pi_i)$ of $\ca{E}_i$ is $\QQ$.
Then $r$ is even, and $\ca{E}_i\simeq \cD_p$
for $i=1,2$ by Corollary \ref{cor:Tate-totally-real}.
Thus the condition (a) holds. 

Conversely,
assume that the condition (a) holds.
In this case, since 
$B\simeq \ca{E}_1\simeq\ca{E}_2$,
we have $\ca{E} \simeq B\times B$. 
Thus $A$ has QM by $B$
by the diagonal map
$B\rightarrow B\times B\simeq\ca{E}$.
\end{proof}

\subsection*{Acknowledgements}

The first author is supported by JSPS KAKENHI Grant Numbers 
JP16K17578, JP21K03187
and
Research Institute for Science and Technology of Tokyo Denki 
University Grant Number Q20K-01 / Japan.
The second author is supported by JSPS KAKENHI Grant Number JP21K11763
and Grant for Basic Science Research Projects 
from The Sumitomo Foundation Grant Number 200484.

\def\bibname{References}

(Keisuke Arai)
Department of Mathematics,
School of Science and Technology for Future Life,
Tokyo Denki University,
5 Senju Asahi-cho, Adachi-ku, Tokyo 120-8551, Japan

\textit{E-mail address}: \texttt{araik@mail.dendai.ac.jp}

\vspace{3mm}

(Yuuki Takai)
Mathematics, Science, Data Science, and AI Program,
Academic Foundations Programs,
Kanazawa Institute of Technology,
7-1 Ohgigaoka, Nonoichi, Ishikawa 921-8501, Japan

\textit{E-mail address}: \texttt{takai@neptune.kanazawa-it.ac.jp}

\end{document}